\newif\ifpreprint
\preprinttrue

\ifpreprint
    \documentclass{article}
    \usepackage{authblk}
\else
    \documentclass[]{interact}
\fi

\usepackage{ifthen}

\usepackage{graphicx}
\usepackage{mathtools}
\usepackage{setspace}
\usepackage{cases}
\usepackage[capitalise]{cleveref}
\usepackage{pifont}
\usepackage{amsmath,amssymb,amsthm}
\usepackage{comment}
\usepackage{tabularray}
\usepackage{xcolor}
\usepackage{subcaption}
\usepackage[numbers,sort&compress]{natbib}
\bibpunct[, ]{[}{]}{,}{n}{,}{,}
\makeatletter
\def\NAT@def@citea{\def\@citea{\NAT@separator}}
\makeatother

\def\usebfsetcapital{\def\setcapital##1{\mathbf{##1}}}

\usebfsetcapital

\def\setR{\setcapital{R}}

\newcommand{\st}{\mathrm{s.t.}}

\newcommand\indicator{\iota}
\newcommand\condition[1]{\quad \text{#1}}
\newcommand\forallcondition[1]{\condition{for all~$#1$}}
\newcommand\eqand{\quad \text{and} \quad}
\newcommand\otherwise{\text{otherwise}}
\DeclareMathOperator*{\argmax}{argmax}
\DeclareMathOperator*{\argmin}{argmin}
\DeclareMathOperator*{\myliminf}{liminf}
\renewcommand{\liminf}{\myliminf}

\DeclareMathOperator*{\conv}{conv}
\DeclareMathOperator{\dom}{dom}
\DeclareMathOperator{\prox}{\mathbf{prox}}
\DeclareMathOperator{\proj}{\mathbf{proj}}
\DeclareMathOperator{\envelope}{\mathcal{M}}
\DeclareMathOperator{\level}{\mathbf{lev}}
\DeclarePairedDelimiter{\abs}{\lvert}{\rvert}
\DeclarePairedDelimiter{\norm}{\lVert}{\rVert}

\DeclarePairedDelimiter{\set}{\lbrace}{\rbrace}
\DeclarePairedDelimiterX{\Set}[2]{\lbrace}{\rbrace}{#1\mathrel{}\delimsize\vert\mathrel{}#2}
\DeclarePairedDelimiterX{\innerp}[2]{\langle}{\rangle}{#1, #2}
\newcommand{\T}{\top\hspace{-1pt}}
\renewcommand{\thesubsubsection}{(\alph{subsubsection})}
\renewcommand{\theenumi}{(\roman{enumi})}
{}

\newcommand{\ra}[1]{\renewcommand{\arraystretch}{#1}}

\theoremstyle{plain}
\newtheorem{theorem}{Theorem}[section]
\newtheorem{lemma}[theorem]{Lemma}
\newtheorem{corollary}[theorem]{Corollary}
\newtheorem{proposition}[theorem]{Proposition}

\theoremstyle{definition}
\newtheorem{definition}[theorem]{Definition}
\newtheorem{example}[theorem]{Example}

\theoremstyle{remark}
\newtheorem{remark}{Remark}

\crefname{equation}{}{}
\Crefname{equation}{Eq.}{Eqs.}
\crefname{figure}{Figure}{Figures}
\crefname{assumption}{Assumption}{Assumptions}

\begin{document}

\title{New merit functions for multiobjective optimization and their properties}

\ifpreprint
    \author[1]{Hiroki Tanabe}
    \author[2]{Ellen H. Fukuda}
    \author[2]{Nobuo Yamashita}

    \affil[1]{Yahoo Japan Corporation}
    \affil[2]{Kyoto University}
    \affil[ ]{\{tanabe.hiroki.45n@kyoto-u.jp\},\{ellen,nobuo\}@i.kyoto-u.ac.jp}

    \date{}
\else
    \author{
        \name{Hiroki Tanabe\textsuperscript{a}\thanks{Hiroki Tanabe Email: tanabe.hiroki.45n@kyoto-u.jp}, Ellen H. Fukuda\textsuperscript{b}\thanks{Ellen H. Fukuda Email: ellen@i.kyoto-u.ac.jp}, and Nobuo Yamashita\textsuperscript{b}\thanks{Nobuo Yamashita Email: nobuo@i.kyoto-u.ac.jp}}
        \affil{\textsuperscript{a}Yahoo Japan Corporation\\
        \textsuperscript{b}Kyoto University}
    }
\fi

\maketitle

\begin{abstract}
    A merit (gap) function is a map that returns zero at the solutions of problems and strictly positive values otherwise.
    Its minimization is equivalent to the original problem by definition, and it can estimate the distance between a given point and the solution set.
    Ideally, this function should have some properties, including the ease of computation, continuity, differentiability, boundedness of the level set, and error boundedness.
    In this work, we propose new merit functions for multiobjective optimization with lower semicontinuous objectives, convex objectives, and composite objectives, and we show that they have such desirable properties under reasonable assumptions.
\end{abstract}

\ifpreprint
\else
    
    \begin{amscode}
        90C29; 90C30
    \end{amscode}
\fi

\section{Introduction}
\label{intro}
Multiobjective optimization is an important field of research with many practical applications.
It minimizes several objective functions at once, but usually, there does not exist a single point that minimizes all objective functions simultaneously.
Therefore, we use the concept of \emph{Pareto optimality}.
We call a point Pareto optimal if there does not exist another point with the same or smaller objective function values and with at least one objective function value being strictly smaller.
However, for non-convex problems, it is difficult to get Pareto optimal solutions. Thus, we use the concept of \emph{Pareto stationarity}. A point is called Pareto stationary if there does not exist a descent direction from it.

Many algorithms for getting Pareto optimal or Pareto stationary solutions have been developed, including the scalarization approaches~\cite{Gass1955,Geoffrion1968,Zadeh1963}, the metaheuristics~\cite{Gandibleux2004}, and the descent methods~\cite{Fliege2009,Fliege2000,Fukuda2014,Tanabe2019}.
However, from a practical point of view (e.g., estimating convergence rates and describing the duality gap), it is essential to know how far a feasible point is from the Pareto set.
Within the context of optimization and equilibrium problems, one of the most commonly used tools to meet such needs is the \emph{merit (gap) functions}~\cite{Auslender1976,Hearn1982}, which return zero at the problem's solutions and strictly positive values otherwise.
Merit functions were first proposed by Auslender~\cite{Auslender1976} in 1976 for variational inequality problems, and became widely known after Hearn~\cite{Hearn1982} re-proposed the same functions and named them \emph{gap functions} when studying the dual gap for convex programming problems in 1982.
If we can minimize the merit function globally, we can obtain the solution of the original problem.
For this reason, the merit functions should have the following properties:
\begin{itemize}
    \item Their values at given points can be quickly evaluated;
    \item Continuity;
    \item (Directional) differentiability;
    \item Their stationary points solve the original problem;
    \item Level-boundedness, i.e., their level sets are bounded;
    \item They provide error bounds, i.e., they are lower bounded by some multiple of the distance between a given point and the solution set of the original problem.
\end{itemize}
It is worth commenting that merit functions have been extensively studied for more general problems such as quasi-variational inequalities~\cite{Franco1995,Fukushima2007,Gupta2012,Mirzaee2017}.
For a comprehensive survey of merit functions for variational inequality and complementarity problems, we refer the reader to the work of Fukushima~\cite{Fukushima1996}.

This paper is not the first attempt to develop merit functions for multiobjective or vector problems.
However, such studies are relatively new compared to the history of research on merit functions for single-objective or scalar problems~\cite{Auslender1976,Fukushima1996,Harker1990}.
First, in 1998, Chen, Goh, and Yang~\cite{Chen1998} developed the merit function for polyhedral-constrained convex multiobjective optimization.
Afterward, various merit functions were considered for vector variational inequalities~\cite{Chen2000,Konnov2005,Li2005,Yang2002,Yang2003,Altangerel2007,Charitha2010,Li2010} and vector equilibrium problems~\cite{Huang2007,Li2005,Li2007,Li2006,Mastroeni2003}.
In 2010, Li and Mastroeni~\cite{Li2010} introduced gap functions with error bounds, and Charitha and Dutta~\cite{Charitha2010} studied regularized gap functions and D-gap functions with continuity, directional differentiability, and error bounds, both for finite-dimensional convex-constrained vector variational inequalities.
On the other hand, the merit functions with error bounds for convex-constrained multiobjective or vector optimization were considered in 2009 for linear objectives~\cite{Liu2009} and in 2017 for convex objectives~\cite{Dutta2017}.
Also, Soleimani-damaneh~\cite{Soleimani-damaneh2008} defines the merit function for multiobjective optimization problems in Banach space, and discusses differentiability by Clarke's generalized gradients.
However, most of those papers do not discuss all of the desired properties mentioned above.
In particular, few discussions exist, related to the computing methods of the merit functions.

Let us now consider the following multiobjective optimization problem:
\begin{equation} \label{eq: CMOP}
    \min_{x \in S} \quad F(x)
,\end{equation} 
where~$F \colon S \to \setR^m$ is a vector-valued function with~$F \coloneqq (F_1, \dots, F_m)^\T$, and~$S \subseteq \setR^n$ is nonempty, closed, and convex.
Here, we propose the following three merit functions for~\cref{eq: CMOP}: a simple one for lower semicontinuous problems, a regularized one for convex problems, and a regularized and partially linearized one for composite problems, i.e., problems with each objective being the sum of a differentiable but not necessarily convex function and a convex but not necessarily differentiable one.
In \cref{tab: merit properties}, we summarize the properties of those merit functions, which will be shown in the subsequent sections.
There, `Sol.' represents the types of Pareto solutions for~\cref{eq: CMOP} corresponding to the minima (zero points) of the merit functions.
Moreover, `SP,' `LB,' and `EB' indicate each~$F_i$'s sufficient conditions so that stationary points of the merit functions can solve~\cref{eq: CMOP}, the merit functions are level-bounded, and the merit functions provide error bounds, respectively.
The simple one connects its minima and the weak Pareto solutions of~\cref{eq: CMOP} but does not have good properties in other aspects.
The regularized one has better properties but requires the convexity of~$F_i$.
The convexity assumption is relaxed in the regularized and partially linearized one, which is also easy to compute for particular problems.
\begin{table}[htpb]
    \caption{Properties of our proposed merit functions}
    \label{tab: merit properties}
    \begin{subtable}{\textwidth}
        \centering
        \caption{Proposed merit functions and their properties}
        \begin{tblr}{hline{1,2,5}={solid}, hline{3,4}={dashed}, width=\textwidth, colspec={@{}X[4,l]X[c]X[c]X[c]X[c]X[c]X[2,c]X[c]@{}}, rowspec={Q[m]Q[m,2em]Q[m]Q[m]}, column{even}={gray!20}}
        & Obj. & Sol. & Cont. & Diff. & SP & LB & EB \\
            Simple & Cont. & \SetCell[r=2]{m}WPO & LSC & $\times$ & $\times$ & LB &\SetCell[r=3]{m} PL \\
            Regularized & Conv. & & \SetCell[r=2]{m}Cont. &\SetCell[r=2]{m} DD & SC & Conv., LB & \\
            Regularized and partially linearized & Comp. & PS & & & SC, $C^2$ & Conv., LB, etc. & \\
        \end{tblr}
    \end{subtable}

    \bigskip
    \begin{subtable}{\textwidth}
        \centering
        \ra{1.3}
        \caption{Table of abbreviations}
        \begin{tblr}{@{}cc@{}}
            \hline
            Obj. & Objective functions \\
            Sol. & Solutions \\
            Cont. & Continuity \\
            Diff. & Differentiability \\
            SP & Stationary points \\
            LB & Level-boundedness \\
            EB & Error bounds \\
            Cont. & Continuity \\
            Comp. & Compositeness \\
            WPO & Weak Pareto optimality \\
            PS & Pareto stationarity \\
            LSC & Lower semicontinuity \\
            DD & Directional differentiability \\
            SC & Strict convexity \\
            $C^2$ & Twice continuous differentiability \\
            PL & Multiobjective proximal-PL inequality \\
            \hline
        \end{tblr}
    \end{subtable}
\end{table}

The outline of this paper is as follows.
In \cref{sec: preliminaries}, we introduce some notations and concepts used in the subsequent discussion.
\Cref{sec: merit} proposes different merit functions for multiobjective optimization with lower semicontinuous objectives, convex objectives, and composite objectives, along with methods for evaluating the function values, the differentiability, and the stationary point properties.
Furthermore, sufficient conditions for them to be level-bounded and to provide error bounds are given in \cref{sec: level-bounded,sec: error bound}, respectively.

\section{Preliminaries}\label{sec: preliminaries}
\subsection{Notations and definitions}
Let us present some notions and definitions used in this paper.
We use the symbol~$\norm*{\cdot}$ for the Euclidean norm in~$\setR^n$.
For~$u, v \in \setR^n$, the notation~$u \le v \, (u < v)$ means that~$u_i \le v_i \, (u_i < v_i)$ for all~$i = 1, \dots, m$.
The zero vector is denoted by~$0$ without mentioning the dimension.
We also define the standard simplex~$\Delta^m \subseteq \setR^m$ by
\begin{equation} \label{eq: simplex}
    \Delta^m \coloneqq \Set*{\lambda \in \setR^m}{\lambda \ge 0 \eqand \sum_{i = 1}^{m} \lambda_i = 1}
.\end{equation} 
Let~$S \subseteq \setR^n$ be a nonempty closed convex set, and let $x \in S$. The \emph{normal cone} of $S$ at $x$, denoted by $N_S(x)$, is defined as
\begin{equation} \label{eq: normal cone}
    N_S(x) \coloneqq \Set*{z \in \setR^n}{z^\T (y - x) \le 0, y \in S}.
\end{equation}
Furthermore, the \emph{convex hull} of a set $A \subseteq \setR^n$, denoted by $\conv(A)$, is the smallest convex set containing $A$.

Now, we introduce some definitions of functions.
A function $h \colon S \to \setR \cup \set{\infty}$ is \emph{lower semicontinuous} at~$x \in S$ if~$h(x) \le \liminf_{k \to \infty} h(x^k)$ for any sequence $\set{x^k} \subseteq S$ convergent to $x$.
In particular, if $h$ is lower semicontinuous at every point on $S$, we say that $h$ is lower semicontinuous or \emph{closed} on $S$.
On the other hand, $h$ is \emph{proper} if its effective domain, defined by $\dom(h) \coloneqq \Set{x \in S}{h(x) < \infty}$, is not empty.
Moreover, we call
\[
    h^\prime(x; d) \coloneqq \lim_{t \searrow 0} \frac{h(x + t d) - h(x)}{t}
\]
the directional derivative of $h \colon S \subseteq \setR^n \to \setR \cup \{\infty\}$ at~$x \in S$ with~$h(x) < \infty$ in the direction~$d \in \setR^n$. Note that $h^\prime(x; d) = \nabla h(x)^\T d$ when $h$ is differentiable at $x$, where $\nabla h(x)$ stands for the gradient of $h$ at $x$, and~$\T$ denotes transpose.
In addition, given~$\sigma > 0$, we say that~$h$ is~$\sigma$-convex if
\[
    h(\alpha x + (1 - \alpha)y) \le \alpha h(x) + (1 - \alpha)h(y) - \frac{\alpha (1 - \alpha) \sigma}{2} \norm*{x - y}^2
\] 
for all~$x, y \in S$ and~$\alpha \in [0, 1]$.
In particular,~$0$-convexity is equivalent to the usual convexity, and when~$\sigma > 0$,~$h$ is called strongly convex.
For a convex function $h \colon S \to \setR \cup \set{\infty}$, we define the subdifferential of~$h$ at~$x \in S$ as
\begin{equation} \label{eq: subdifferential}
    \partial h(x) = \Set*{h' \in \setR^n}{h(y) \ge h(x) + (h')^\T (y - x) \forallcondition{y \in S}}.
\end{equation}

We now suppose that~$h \colon \setR^n \to \setR \cup \set{\infty}$ is a closed, proper, and convex function.
Then, the \emph{Moreau envelope} or \emph{Moreau-Yosida regularization}~$\envelope_h \colon \setR^n \to \setR$ is given by
\begin{equation} \label{eq: Moreau envelope}
    \envelope_h(x) \coloneqq \min_{y \in \setR^n} \left\{ h(y) + \frac{1}{2} \norm*{x - y}^2 \right\} 
.\end{equation} 
The minimization problem in~\cref{eq: Moreau envelope} has a unique solution because of the strong convexity of its objective functions.
By this solution, the \emph{proximal operator} is defined as
\begin{equation} \label{eq: proximal operator}
    \prox_h(x) \coloneqq \argmin_{y \in \setR^n} \left\{ h(y) + \frac{1}{2} \norm*{x - y}^2 \right\} 
.\end{equation} 
The proximal operator is non-expansive, i.e.,~$\norm*{\prox_h(x) - \prox_h(y)} \le \norm*{x - y}$.
This also means that~$\prox_h$ is $1$-Lipschitz continuous.
Moreover, when~$h$ is the \emph{indicator function} of~$C \subseteq \setR^n$, i.e,
\begin{equation} \label{eq: indicator function}
    \indicator_C(x) \coloneqq \begin{dcases}
        0, & x \in C, \\
        \infty, & x \notin C
    ,\end{dcases}
\end{equation} 
where~$C \subseteq \setR^n$ is a nonempty closed convex set, the proximal operator of~$h$ reduces to the projection onto~$C$, i.e.,
\begin{equation} \label{eq: proximal operator of the indicator function}
    \prox_{\indicator_C}(x) = \proj_C(x) \coloneqq \argmin_{y \in C} \norm*{x - y}
.\end{equation} 
Recall that~$h$ is closed, proper, and convex.
Even if~$h$ is non-differentiable, its Moreau envelope~$\envelope_h$ is known to be differentiable.
\begin{theorem}~\cite[Theorem 6.60]{Beck2017} \label{thm: smoothness of Moreau envelope}
    Let~$h \colon \setR^n \to \setR \cup \set*{\infty}$ be a closed, proper, and convex function.
    Then,~$\envelope_h$ has an~$1$-Lipschitz continuous gradient given by
    \[
        \nabla \envelope_h(x) = x - \prox_h(x)
    .\] 
\end{theorem}
We also refer to the so-called second prox theorem as well as a corollary quickly derived from it.
\begin{theorem}~\cite[Theorem 6.39]{Beck2017} \label{thm: second prox}
    Let~$h \colon \setR^n \to \setR \cup \set*{\infty}$ be a closed, proper, and convex function.
    Then, for any~$x \in \setR^n$, we have
    \[
        \left( x - \prox_h(x) \right)^\T \left( y - \prox_h(x) \right) \le h(y) - h\left( \prox_h(x) \right) \forallcondition{y \in \setR^n}
    .\] 
\end{theorem}
\begin{corollary} \label{cor: second prox}
    Let~$h \colon \setR^n \to \setR \cup \set*{\infty}$ be a closed, proper, and convex function.
    Then, it follows that
    \[
        \norm*{x - \prox_h(x)}^2 \le h(x) - h \left( \prox_h(x) \right) \forallcondition{x \in \setR^n}
    .\] 
\end{corollary}

We move on to H\"older and Lipschitz continuities.
We call~$h \colon \setR^n \to \setR$ to be \emph{locally H\"older continuous} with exponent~$\beta > 0$ if for every bounded set~$\Omega \subseteq \setR^n$ there exists~$L > 0$ such that
\[
    \abs{h(x) - h(y)} \le L \norm{x - y}^\beta \forallcondition{x, y \in \Omega}
.\]
In particular, when~$L$ does not depend on~$\Omega$, we say that~$h$ is H\"older continuous with exponent~$\beta > 0$.
Moreover, we refer to the (local) H\"older continuity with exponent~$1$ as the \emph{(local) Lipschitz continuity}.
When~$h$ is Lipschitz continuous, we call~$L$ the \emph{Lipschitz constant}, and we also say that~$h$ is \emph{$L$-Lipschitz continuous}.
As the following lemma shows, many functions with \emph{good} properties are locally Lipschitz continuous.
\begin{lemma} \label{thm:local_Lipschitz}
    Continuously differentiable functions and finite-valued convex functions are locally Lipschitz continuous.
\end{lemma}
\begin{proof}
    The former is due to the mean value theorem, and the latter is from~\cite{WayneStateUniversity1972}.
\end{proof}

Finally, we recall a fact on sensitivity analysis for the following parameterized optimization problem:
\begin{equation} \label{eq: parameterized optimization problem}
    \min_{x \in X} h(x, \xi)
,\end{equation}
depending on the parameter vector~$\xi \in \Xi$.
Here,~$h \colon X \times \Xi \subseteq \setR^p \times \setR^q \to \setR \cup \set{\infty}$ is the objective function.
We assume that~$X \subseteq \setR^p$ and~$\Xi \subseteq \setR^q$ are nonempty and closed.
Let us write the optimal value function of~\cref{eq: parameterized optimization problem} as
\begin{equation} \label{eq: optimal value function}
    \phi(\xi) \coloneqq \inf_{x \in X} h(x, \xi)
\end{equation} 
and the associated set as
\begin{equation} \label{eq: optimal solution mapping}
    \Phi(\xi) \coloneqq \Set*{x \in X}{\phi(\xi) = h(x, \xi)}
.\end{equation} 
The following proposition describes the directional differentiability of the optimal value function~$\phi$.
\begin{proposition}~\cite[Proposition 4.12]{Bonnans2000} \label{prop: directional differentiability}
    Let~$\xi^0 \in \Xi$.
    Suppose that
    \begin{enumerate}
        \item the function~$h(x, \xi)$ is continuous on~$X \times \Xi$; \label{prop: directional differentiability: continuity}
        \item there exist~$\alpha \in \setR$ and a compact set~$C \subseteq X$ such that for every~$\hat{\xi}$ near~$\xi^0$, the level set~$\level_\alpha h(\cdot, \hat{\xi})$ is nonempty and contained in~$C$; \label{prop: directional differentiability: level set}
        \item for any~$x \in X$ the function~$h_x(\cdot) \coloneqq h(x, \cdot)$ is directionally differentiable at~$\xi^0$; \label{prop: directional differentiability: directional differentiability}
        \item if~$\xi \in \Xi$,~$t_k \searrow 0$, and~$\set*{x^k}$ is a sequence in~$C$ given by~\ref{prop: directional differentiability: level set}, then~$\set*{x^k}$ has a limit point~$\bar{x}$ such that
            \[
                \limsup_{k \to \infty} \frac{h(x^k, \xi^0 + t_k (\xi - \xi^0)) - h(x^k, \xi^0)}{t_k} \ge h_{\bar{x}}^\prime(\xi^0; \xi - \xi^0)
            .\] \label{prop: directional differentiability: limit}
    \end{enumerate}
    Then, the optimal value function~$\phi$ given by~\cref{eq: optimal value function} is directionally differentiable at~$\xi^0$ and
    \[
        \phi^\prime(\xi^0; \xi - \xi^0) = \inf_{x \in \Phi(\xi^0)} h_x^\prime(\xi^0; \xi - \xi^0)
    .\] 
\end{proposition}

\subsection{Optimality, stationarity, and level-boundedness for multiobjective optimization}
We first introduce the concept of optimality and stationarity for~\cref{eq: CMOP}.
Note that vector-to-vector inequalities are componentwise, as defined at the beginning of the previous section.
Recall that $x^\ast \in S$ is \emph{Pareto optimal} if there is no $x \in S$ such that $F(x) \le F(x^\ast)$ and $F(x) \neq F(x^\ast)$.
Likewise, $x^\ast \in S$ is \emph{weakly Pareto optimal} if there does not exist $x \in S$ such that $F(x) < F(x^\ast)$. It is known that Pareto optimal points are always weakly Pareto optimal, and the converse is not true necessarily. When~$F$ is directionally differentiable, we also say that $\bar{x} \in S$ is \emph{Pareto stationary}~\cite{Tanabe2019} if
\[
    \max_{i = 1, \dots, m} F_i^\prime(\bar{x}; z - \bar{x}) \ge 0 \quad \text{for all } z \in S.
\]
We state below the relation between the three concepts of Pareto optimality.
\begin{lemma}~\cite[Lemma 2.2]{Tanabe2019} \label{lem: Pareto relationship}
    When~$F$ is directionally differentiable, the following three statements hold.
    \begin{enumerate}
        \item If $x \in S$ is weakly Pareto optimal for~\eqref{eq: CMOP}, then $x$ is Pareto stationary.\label{enum: weak stationary}
        \item Let every component~$F_i$ of $F$ be convex. If $x \in S$ is Pareto stationary for~\eqref{eq: CMOP}, then $x$ is weakly Pareto optimal.\label{enum: stationary weak}
        \item Let every component~$F_i$ of $F$ be strictly convex. If $x \in S$ is Pareto stationary for~\eqref{eq: CMOP}, then $x$ is Pareto optimal.\label{enum: stationary pareto}
    \end{enumerate}
\end{lemma}

Now, let us extend the level-boundedness~\cite[Definition 1.8]{Rockafellar1998} for a scalar-valued function~$f \colon S \to \setR$, i.e.,~$\level_\alpha f \coloneqq \Set*{x \in S}{f(x) \le \alpha}$ is bounded for any~$\alpha \in \setR$, to a vector-valued function as follows.
\begin{definition}
    A vector-valued function~$F \colon S \to \setR^m$ is level-bounded if the level set~$\level_\alpha F \coloneqq \Set{x \in S}{F(x) \le \alpha}$ is bounded for all~$\alpha \in \setR^m$.
\end{definition}
If $F_i$ is level-bounded for all $i = 1, \dots, m$, then $F = (F_1, \dots, F_m)^\T$ is also level-bounded.
Note that even if $F$ is level-bounded, every $F_i$ is not necessarily level-bounded (e.g. $F(x) = (x - 1, - 2 x + 1)^\T$).
Now, we show the existence of weakly Pareto optimal points under the level-boundedness assumption.
Recall that throughout the paper, the feasible set~$S$ is nonempty, closed, and convex.
\begin{theorem}
    If $F$ is closed and level-bounded, then \eqref{eq: CMOP} has a weakly Pareto optimal solution.
\end{theorem}
\begin{proof}
    Let $F$ be closed and level-bounded.
    Then the level set~$\level_\alpha F \coloneqq
    \Set{ x \in S}{F_i(x) \le \alpha \text{ for all }
    i = 1, \dots, m}$ is bounded for all~$\alpha \in \setR$.
    Now, we have
    \[
        \level_\alpha F = \Set*{ x \in S }{\max_{i = 1, \dots, m} F_i(x) \le \alpha} = \level_\alpha \left( \max_{i = 1, \dots, m} F_i \right),
    \]
    so~$\max_i F_i$ is also level-bounded.
    Moreover, since~$F_i$ is closed for each~$i = 1, \dots, m$,~$\max_i F_i$ is also closed.
    Thus, the problem
    \begin{align*}
        \min \quad & \max_{i = 1, \dots, m} F_i(x) \\
        \st \quad  & x \in S
    \end{align*}
    has a global optimal solution~$x^\ast$. This gives
    \[
        \quad \max_{i = 1, \dots, m} F_i(x^\ast)
            \le \max_{i = 1, \dots, m} F_i(x) \quad \forallcondition{x \in S}.
    \]
    Since $\max_{i = 1, \dots, m} a_i - \max_{i = 1, \dots, m} (a_i - b_i) \le \max b_i$, we have $\max_{i = 1, \dots, m} F_i(x) - \max_{i = 1, \dots, m} (F_i(x) - F_i(x^*)) \le \max_{i = 1, \dots, m} F_i(x^*)$ for all $x \in S$, which together with the above inequality gives $\max_{i = 1, \dots, m} (F_i(x^*) - F_i(x)) \le 0$ for all $x \in S$.
    As this means $F_i(x^*) \le F_i(x)$ for all~$i = 1, \dots, m$ and all~$x \in S$, we obtain the result.
\end{proof}

\section{New merit functions for multiobjective optimization} \label{sec: merit}
A merit function associated with an optimization problem is a function that returns zero at their solutions and strictly positive values otherwise, which implies that it is nonnegative~\cite{Auslender1976,Hearn1982}.
This section proposes different types of merit functions for the multiobjective optimization problem~\cref{eq: CMOP}, considering three cases, respectively, when the objective function~$F$ is lower semicontinuous, when it is convex, and when it has a composite structure.

\subsection{A simple merit function for lower semicontinuous multiobjective optimization} \label{sec: merit continuous}
First, we assume only lower semicontinuity on~$F$ and propose a simple merit function~$u_0 \colon S \to \setR \cup \set*{\infty}$ as follows:
\begin{equation} \label{eq: u_0}
    u_0(x) \coloneqq \sup_{y \in S} \min_{i = 1, \dots, m} \left\{ F_i(x) - F_i(y) \right\} 
.\end{equation}
When~$F$ is linear, this merit function has already been discussed in~\cite{Liu2009}, but here we consider the more general nonlinear cases.
We now show that $u_0$ is a merit function in the sense of weak Pareto optimality.
\begin{theorem} \label{thm: u}
    Let~$u_0$ be defined by~\cref{eq: u_0}.
    Then, we have~$u_0(x) \ge 0$ for all~$x \in S$.
    Moreover,~$x \in S$ is weakly Pareto optimal for~\cref{eq: CMOP} if and only if~$u_0(x) = 0$.
\end{theorem}
\begin{proof}
    Let~$x \in S$.
    By the definition~\cref{eq: u_0} of~$u_0$, we get
    \[
        u_0(x) = \sup_{y \in S} \min_{i = 1, \dots, m} \left\{ F_i(x) - F_i(y) \right\}
                 \ge \min_{i = 1, \dots, m} \left\{ F_i(x) - F_i(x) \right\} = 0
    .\] 
    On the other hand, again considering the definition~\cref{eq: u_0} of~$u_0$, we obtain
    \[
        u_0(x) = 0 \iff \min_{i = 1, \dots, m} \left\{ F_i(x) - F_i(y) \right\} \le 0 \text{ for all } y \in S
    .\] 
    So, there does not exist~$y \in S$ such that
    \[
        F_i(x) - F_i(y) > 0 \forallcondition{i = 1, \dots, m}
    ,\] 
    which means that~$x$ is weakly Pareto optimal for~\cref{eq: CMOP} by definition.
\end{proof}

The following theorem is clear from the lower semicontinuity of~$F_i$ and~\cite[Theorem~10.3]{Rooij1982}.
\begin{theorem} \label{thm: u continuity}
    The function~$u_0$ defined by~\cref{eq: u_0} is lower semicontinuous on~$S$.
\end{theorem}
\Cref{thm: u continuity,thm: u} imply that if~$u_0(x^k) \to 0$ holds for some bounded sequence~$\set*{x^k}$, its accumulation points are weakly Pareto optimal.
Thus, we can use~$u_0$ to measure the convergence rate of multiobjective optimization methods (i.e.,~\cite{Tanabe2023}).

Moreover, \cref{thm: u} implies that we can get weakly Pareto optimal solutions via the following single-objective optimization problem:
\[
    \min_{x \in S} \quad u_0(x)
.\] 
However, in some cases, such as when~$F_i$ is not bounded from below on~$S$ for all~$i = 1, \dots, m$, we cannot guarantee that~$u_0$ is finite-valued.
Moreover, even if~$u_0$ is finite-valued,~$u_0$ does not preserve the differentiability of the original objective function~$F$.

\subsection{A regularized merit function for convex multiobjective optimization} \label{sec: merit convex}
Here, we suppose that each component~$F_i$ of the objective function~$F$ of~\cref{eq: CMOP} is convex.
Then, we define a regularized merit function~$u_\ell \colon S \to \setR$ with a given constant~$\ell > 0$, which overcomes the shortcomings mentioned at the end of the previous subsection, as follows:
\begin{equation} \label{eq: u_ell}
    u_\ell(x) \coloneqq \max_{y \in S} \min_{i = 1, \dots, m} \left\{ F_i(x) - F_i(y) - \frac{\ell}{2}\norm*{x - y}^2 \right\} 
.\end{equation} 
Note that the strong concavity of the function inside~$\max_{y \in S}$ and~\cite[Theorem~2.2.6]{Nesterov2004} imply that there exists a unique solution~$U_\ell(x) \in S$ given by
\begin{equation} \label{eq: U_ell}
    U_\ell(x) \coloneqq \argmax_{y \in S} \min_{i = 1, \dots, m} \left\{ F_i(x) - F_i(y) - \frac{\ell}{2}\norm*{x - y}^2 \right\} 
\end{equation}
and~$u_\ell$ is finite-valued.
Like~$u_0$, we can show that~$u_\ell$ is also a merit function in the sense of weak Pareto optimality.
\begin{theorem} \label{thm: u ell}
    Let~$u_\ell$ be defined by~\cref{eq: u_ell} for some~$\ell > 0$.
    Then, we have~$u_\ell(x) \ge 0$ for all~$x \in S$.
    Moreover,~$x \in S$ is weakly Pareto optimal for~\cref{eq: CMOP} if and only if~$u_\ell(x) = 0$.
\end{theorem}
\begin{proof}
    Let~$x \in S$.
    The definition~\cref{eq: u_ell} of~$u_\ell$ yields
    \[
        \begin{split}
            u_\ell(x) &= \max_{y \in S} \min_{i = 1, \dots, m} \left\{ F_i(x) - F_i(y) - \frac{\ell}{2} \norm*{x - y}^2 \right\} \\
                      &\ge \min_{i = 1, \dots, m} \left\{ F_i(x) - F_i(x) - \frac{\ell}{2} \norm*{x - x}^2 \right\} = 0
        ,\end{split}
    \] 
    which proves the first statement.

    We now prove the second statement.
    First, assume that~$u_\ell(x) = 0$.
    Then, \cref{eq: u_ell} again gives
    \[
        \min_{i = 1, \dots, m} \left\{ F_i(x) - F_i(y) - \frac{\ell}{2} \norm*{x - y}^2 \right\} \le 0 \forallcondition{y \in S}
    .\] 
    Let~$z \in S$ and~$\alpha \in (0, 1)$.
    Since the convexity of~$S$ implies that~$x + \alpha (z - x) \in S$, by substituting~$y = x + \alpha (z - x)$ into the above inequality, we get
    \[
        \min_{i = 1, \dots, m} \left\{ F_i(x) - F_i(x + \alpha(z - x)) - \frac{\ell}{2} \norm*{\alpha(z - x)}^2 \right\} \le 0
    .\] 
    The convexity of~$F_i$ leads to
    \[
        \min_{i = 1, \dots, m} \left\{ \alpha (F_i(x) - F_i(z)) - \frac{\ell}{2} \norm*{\alpha (z - x)}^2 \right\} \le 0
    .\] 
    Dividing both sides by~$\alpha$ and letting~$\alpha \searrow 0$, we have
    \[
        \min_{i = 1, \dots, m} \left\{ F_i(x) - F_i(z) \right\} \le 0
    .\] 
    Since~$z$ can take an arbitrary point in~$S$, it follows from~\cref{eq: u_0} that~$u_0(x) = 0$.
    Therefore, from \cref{thm: u},~$x$ is weakly Pareto optimal.

    Now, suppose that~$x$ is weakly Pareto optimal.
    Then, it follows again from \cref{thm: u} that~$u_0(x) = 0$.
    It is clear that~$u_\ell(x) \le u_0(x)$ from the definitions~\cref{eq: u_0,eq: u_ell} of~$u_0$ and~$u_\ell$.
    So, we get~$u_\ell(x) = 0$.
\end{proof}
Then, we can also show the continuity of~$u_\ell$ and~$U_\ell$ without any particular assumption.
\begin{theorem} \label{thm: u_ell continuity}
    For each~$\ell > 0$,~$u_\ell$ and~$U_\ell$ defined by~\cref{eq: u_ell,eq: U_ell} are locally Lipschitz continuous and locally H\"older continuous with exponent~$1 / 2$ on~$S$, respectively.
\end{theorem}
\begin{proof}
    The optimality condition of the maximization problem associated with~\cref{eq: u_ell,eq: U_ell} and~\cite[Proposition~A.22]{Bertsekas1971} give
    \[
        \ell \left( x - U_\ell(x) \right) \in \conv_{i \in \mathcal{I}(x)} \partial F_i(U_\ell(x)) + N_S(U_\ell(x)) \forallcondition{x \in S}
    ,\]
    where~$N_S$ denotes the normal cone to the convex set~$S$ defined by~\cref{eq: normal cone} and
    \[
        \mathcal{I}(x) = \argmin_{i = 1, \dots, m} \left\{ F_i(x) - F_i(U_\ell(x)) \right\}
    .\]
    Considering the definitions~\cref{eq: normal cone,eq: subdifferential} of the normal cone and subdifferential, for each~$x \in S$ there exists~$\lambda(x) \in \Delta^m$, where~$\Delta^m$ is the standard simplex given by~\cref{eq: simplex}, such that~$\lambda_j(x) = 0$ for all~$j \notin \mathcal{I}(x)$ and
    \[
        \ell (x - U_\ell(x))^\T (z - U_\ell(x)) \le \sum_{i = 1}^{m} \lambda_i(x) \left\{ F_i(z) - F_i(U_\ell(x)) \right\} \forallcondition{z \in S}
    .\] 
    For any bounded set~$\Omega \subseteq S$, let~$x^1, x^2 \in \Omega$.
    Adding the two inequalities obtained by substituting~$(x, z) = (x^1, U_\ell(x^2))$ and~$(x, z) = (x^2, U_\ell(x^1))$ into the above inequality, we get
    \begin{align*}
        \MoveEqLeft \ell \left(U_\ell(x^1) - U_\ell(x^2) - (x^1 - x^2)\right)^\T (U_\ell(x^1) - U_\ell(x^2)) \\
        \le{}& \sum_{i = 1}^{m} (\lambda_i(x^2) - \lambda_i(x^1)) \left\{ F_i(U_\ell(x^1)) - F_i(U_\ell(x^2)) \right\} \\
        ={}& \sum_{i = 1}^{m} \lambda_i(x^1) \left\{ F_i(x^1) - F_i(U_\ell(x^1)) \right\} + \sum_{i = 1}^{m} \lambda_i(x^2) \left\{ F_i(x^2) - F_i(U_\ell(x^2)) \right\} \\
           &+ \sum_{i = 1}^{m} \lambda_i(x^1) \left\{ F_i(U_\ell(x^2)) - F_i(x^1) \right\} + \sum_{i = 1}^{m} \lambda_i(x^2) \left\{ F_i(U_\ell(x^1)) - F_i(x^2) \right\} \\
        ={}& \min_{i = 1, \dots, m} \left\{ F_i(x^1) - F_i(U_\ell(x^1)) \right\} + \min_{i = 1, \dots, m} \left\{ F_i(x^2) - F_i(U_\ell(x^2)) \right\} \\
           &+ \sum_{i = 1}^{m} \lambda_i(x^1) \left\{ F_i(U_\ell(x^2)) - F_i(x^1) \right\} + \sum_{i = 1}^{m} \lambda_i(x^2) \left\{ F_i(U_\ell(x^1)) - F_i(x^2) \right\} \\
        \le{}& \sum_{i = 1}^{m} \lambda_i(x^2) \left\{ F_i(x^1) - F_i(U_\ell(x^1)) \right\} + \sum_{i = 1}^{m} \lambda_i(x^1) \left\{ F_i(x^2) - F_i(U_\ell(x^2)) \right\} \\
           &+ \sum_{i = 1}^{m} \lambda_i(x^1) \left\{ F_i(U_\ell(x^2)) - F_i(x^1) \right\} + \sum_{i = 1}^{m} \lambda_i(x^2) \left\{ F_i(U_\ell(x^1)) - F_i(x^2) \right\} \\
        ={}& \sum_{i = 1}^{m} \left( \lambda_i(x^2) - \lambda_i(x^1) \right) \left\{ F_i(x^1) - F_i(x^2) \right\} 
        \le 2 \max_{i = 1, \dots, m} \abs*{F_i(x^1) - F_i(x^2)}
    ,\end{align*}
    where the second equality holds from the definition of~$\mathcal{I}(x)$ and since~$\lambda(x) \in \Delta^m$ and~$\lambda_j(x) \neq 0$ for all~$j \in \mathcal{I}(x)$.
    Dividing by~$\ell$ and adding~$(1 / 4) \norm*{x^1 - x^2}^2$ in both sides of the inequality, we have
    \[
        \norm*{U_\ell(x^1) - U_\ell(x^2) - \frac{1}{2} \left( x^1 - x^2 \right) }^2 \le \frac{1}{4} \norm*{x^1 - x^2}^2 + \frac{2}{\ell} \max_{i = 1, \dots, m} \abs*{F_i(x^1) - F_i(x^2)}
    .\] 
    Taking the square root of both sides, we obtain
    \[
        \norm*{U_\ell(x^1) - U_\ell(x^2) - \frac{1}{2} \left( x^1 - x^2 \right) } \le \sqrt{\frac{1}{4} \norm*{x^1 - x^2}^2 + \frac{2}{\ell} \max_{i = 1, \dots, m} \abs*{F_i(x^1) - F_i(x^2)}}
    .\] 
    Then, it follows from the triangle inequality that
    \[
        \norm*{U_\ell(x^1) - U_\ell(x^2)} \le \frac{1}{2} \norm*{x^1 - x^2} + \sqrt{\frac{1}{4} \norm*{x^1 - x^2}^2 + \frac{2}{\ell} \max_{i = 1, \dots, m} \abs*{F_i(x^1) - F_i(x^2)}}
    .\] 
    Since~\cref{thm:local_Lipschitz} implies that~$F_i$ is locally Lipschitz continuous on~$S$, there exists~$L_i(\Omega) > 0$ such that
    \begin{equation} \label{eq: u_ell continuity: local_Lipschitz}
        \abs*{F_i(x^1) - F_i(x^2)} \le L_i(\Omega) \norm*{x^1 - x^2}
    \end{equation}
    Hence, the above two inequlities show~$U_\ell$'s local H\"older continuity with exponent~$1 / 2$.

    On the other hand, the definition~\cref{eq: u_ell} of~$u_\ell$ gives
    \begin{align*}
        u_\ell(x^1) & = \max_{y \in C} \min_{i = 1, \dots, m} \left[ F_i(x^1) - F_i(y) - \frac{\ell}{2} \norm*{x^1 - y}^2 \right] \\
                                 & \ge \min_{i = 1, \dots, m} \left[F_i(x^1) - F_i(U_\ell(x^2))\right] - \frac{\ell}{2} \norm*{x^1 - U_\ell(x^2)}^2
        .\end{align*}
    Reducing~$u_\ell(x^2)$ from both sides yields
    \[
        u_\ell(x^1) - u_\ell(x^2) \ge \min_{i = 1, \dots, m} \left[F_i(x^1) - F_i\left(U_\ell(x^2)\right) - \frac{\ell}{2} \norm*{x^1 - U_\ell(x^2)}^2\right] - u_\ell(x^2)
    .\]
    \cref{eq: u_ell,eq: U_ell} lead to
    \begin{align*}
        u_\ell(x^1) - u_\ell(x^2) \ge{} & \min_{i = 1, \dots, m} \left[F_i(x^1) - F_i\left(U_\ell(x^2)\right) - \frac{\ell}{2} \norm*{x^1 - U_\ell(x^2)}^2\right]    \\
                                                                  & - \min_{i = 1, \dots, m} \left[F_i(x^1) - F_i\left(U_\ell(x^2)\right) - \frac{\ell}{2} \norm*{x^2 - U_\ell(x^2)}^2 \right]
    .\end{align*}
    As it follows that that~$\min_{i = 1, \dots, m} v^1_i - \min_{i = 1, \dots, m} v^2_i \ge \min_{i = 1, \dots, m} (v^1_i - v^2_i)$ for any~$v^1, v^2 \in \setR^m$, we obtain
    \[
        u_\ell(x^1) - u_\ell(x^2) \ge \min_{i = 1, \dots, m} \left[ F_i(x^1) - F_i(x^2) - \frac{\ell}{2} \left( x^1 + x^2 - 2 U_\ell(x^2) \right)^\T \left( x^1 - x^2 \right) \right]
    .\]
    Cauchy-Schwarz inequality and~\cref{eq: u_ell continuity: local_Lipschitz} implies
    \[
        u_\ell(x^1) - u_\ell(x^2) \ge - \left[ \max_{i = 1, \dots, m} L_i(\Omega) + \frac{\ell}{2} \norm*{x^1 + x^2 - 2 U_\ell(x^2)}\right] \norm*{x^1 - x^2}
    .\]
    Since~$U_\ell(x)$ is bounded for~$x \in \Omega$ due to the continuity, and the above inequality holds even if we interchange~$x^1$ and~$x^2$, we can show the local Lipschitz continuity of~$u_\ell$.
\end{proof}

On the other hand, since~$\min_{i = 1, \dots, m} q_i = \min_{\lambda \in \Delta^m} \sum_{i = 1}^m q_i$ for any~$q \in \setR^m$ with the standard simplex~$\Delta^m$ defined by~\cref{eq: simplex}, $u_\ell$ given by~\cref{eq: u_ell} can also be expressed as
\[
    u_\ell(x) = \max_{y \in S} \min_{\lambda \in \Delta^m} \sum_{i = 1}^{m} \lambda_i \left\{ F_i(x) - F_i(y) - \frac{\ell}{2} \norm*{x - y}^2 \right\} 
.\] 
We can see that~$S$ is convex,~$\Delta^m$ is compact and convex, and the function inside~$\min_{\lambda \in \Delta^m}$ is convex for~$\lambda$ and concave for~$y$.
Therefore, Sion's minimax theorem~\cite{Sion1958} leads to
\begin{equation} \label{eq: u dual}
    \begin{split}
        u_\ell(x) &= \min_{\lambda \in \Delta^m} \max_{y \in S} \sum_{i = 1}^{m} \lambda_i \left\{ F_i(x) - F_i(y) - \frac{\ell}{2} \norm*{x - y}^2 \right\} \\
                  &= \min_{\lambda \in \Delta^m} \left\{ \sum_{i = 1}^{m} \lambda_i F_i(x) - \ell \envelope_{ \frac{1}{\ell} \sum_{i = 1}^{m} \lambda_i F_i + \indicator_S } (x) \right\}
    ,\end{split}
\end{equation}
where~$\envelope$ and~$\indicator$ denote the Moreau envelope and the indicator function defined by~\cref{eq: Moreau envelope,eq: indicator function}, respectively.
Thus, for each~$\ell > 0$, we can evaluate~$u_\ell$ through the following~$m$-dimensional, simplex-constrained, convex optimization problem:
\begin{equation} \label{eq: dual of u_ell}
    \begin{aligned}
        \min_{\lambda \in \setR^m} &&& \sum_{i = 1}^{m} \lambda_i F_i(x) - \ell \envelope_{\frac{1}{\ell} \sum_{i = 1}^{m} \lambda_i F_i + \indicator_S}(x) \\
        \st &&& \lambda \ge 0 \eqand \sum_{i = 1}^{m} \lambda_i = 1 
    .\end{aligned}
\end{equation} 
As the following theorem shows, the objective function of~\cref{eq: dual of u_ell} is continuously differentiable.
\begin{theorem} \label{thm: differentiability of the objective function of u_ell}
    Let~$x \in S$ be given.
    The objective function of~\cref{eq: dual of u_ell} is continuously differentiable at every~$\lambda \in \setR^m$ and
    \[
        \nabla_\lambda \left[ \sum_{i = 1}^{m} \lambda_i F_i(x) - \ell \envelope_{\frac{1}{\ell} \sum_{i = 1}^{m} \lambda_i F_i + \indicator_S}(x) \right] = F(x) - F\left( \prox_{\frac{1}{\ell} \sum_{i = 1}^{m} \lambda_i F_i + \indicator_S}(x) \right)
    ,\] 
    where~$\prox$ denotes the proximal operator~\cref{eq: proximal operator}.
\end{theorem}
\begin{proof}
    Define
    \[
        h(y, \lambda) \coloneqq \sum_{i = 1}^{m} \lambda_i F_i(y) + \frac{1}{2} \norm*{x - y}^2
    .\] 
    We now check that~$\envelope_{\frac{1}{\ell} \sum_{i = 1}^{m} \lambda_i F_i + \indicator_S}(x) = \min_{y \in S} h(y, \lambda)$ satisfies the assumptions~\ref{prop: directional differentiability: continuity}--\ref{prop: directional differentiability: limit} of~\cref{prop: directional differentiability}.
    First, since~$F$ is finite-valued and convex,~$h$ is clearly continuous on~$S \times \Delta^m$~(Assumption~\ref{prop: directional differentiability: continuity}).
    Moreover,~$h_y(\cdot) \coloneqq h(y, \cdot)$ is continuously differentiable and
    \[
        \nabla_\lambda h_y(\lambda) = F(y)
    \] 
    for any~$y \in S$ (Assumption~\ref{prop: directional differentiability: directional differentiability}).
    Furthermore,~$\prox_{\frac{1}{\ell} \sum_{i = 1}^{m} \lambda_i F_i + \indicator_S}(x) = \argmin_{y \in S} h(y, \lambda)$ is also continuous at every~$\lambda \in \setR^m$ from~\cite[Excercise 7.38]{Rockafellar1998} (Assumptions~\ref{prop: directional differentiability: level set} and~\ref{prop: directional differentiability: limit} with~$C = \set{\prox_{\frac{1}{\ell} \sum_{i = 1}^{m} \lambda_i F_i + \indicator_S}(x)}$).
    Therefore, all assumptions of \cref{prop: directional differentiability} are satisfied.
    Since~$\prox_{\frac{1}{\ell} \sum_{i = 1}^{m} \lambda_i F_i + \indicator_S}(x)$ is unique, we obtain the desired result.
\end{proof}
Therefore, when~$\prox_{\frac{1}{\ell} \sum_{i = 1}^{m} \lambda_i F_i + \indicator_S}(x)$ is easy to compute, we can solve~\cref{eq: dual of u_ell} using well-known convex optimization techniques such as the interior point method~\cite{Bertsekas1999}.
If~$n \gg m$, this is usually faster than solving the~$n$-dimensional problem directly to compute~\cref{eq: u_ell}.

Let us now write the optimal solution set of~\cref{eq: dual of u_ell} by
\begin{equation} \label{eq: Lambda}
    \Lambda(x) \coloneqq \argmin_{\lambda \in \Delta^m} \left\{ \sum_{i = 1}^{m} \lambda_i F_i(x) - \ell \envelope_{ \frac{1}{\ell} \sum_{i = 1}^{m} \lambda_i F_i + \indicator_S } (x) \right\}
.\end{equation} 
Then, we can show the directional differentiability of~$u_\ell$, as in the following theorem.
\begin{theorem} \label{thm: directional differentiability of u_ell}
    Let~$x \in S$.
    For each~$\ell > 0$, the merit function~$u_\ell$ defined by~\cref{eq: u_ell} has a directional derivative
    \begin{multline*}
        u_\ell^\prime(x; z - x) \\= \inf_{\lambda \in \Lambda(x)} \left\{ \sum_{i = 1}^{m} \lambda_i F_i^\prime(x; z - x) - \ell \left(x - \prox_{\frac{1}{\ell} \sum_{i = 1}^{m} \lambda_i F_i + \indicator_S}(x)\right)^\T (z - x) \right\} 
    \end{multline*}
    for all~$z \in S$, where~$\Lambda(x)$ is given by~\cref{eq: Lambda}, and~$\prox$ denotes the proximal operator~\cref{eq: proximal operator}.
    In particular, if~$\Lambda(x)$ is a singleton, i.e.,~$\Lambda(x) = \set*{\lambda(x)}$, and~$F_i$ is continuously differentiable at~$x$, then~$u_\ell$ is continuously differentiable at~$x$, and we have
    \[
        \nabla u_\ell(x) = \sum_{i = 1}^{m} \lambda_i(x) \nabla F_i(x) - \ell \left( x - \prox_{\frac{1}{\ell} \sum_{i = 1}^{m} \lambda_i(x) F_i + \indicator_S} \right) 
    .\] 
\end{theorem}
\begin{proof}
    Let
    \[
        h(x, \lambda) \coloneqq \sum_{i = 1}^{m} \lambda_i F_i(x) - \ell \envelope_{ \frac{1}{\ell} \sum_{i = 1}^{m} \lambda_i F_i + \indicator_S } (x)
    .\]
    Since~$\envelope_{ \frac{1}{\ell} \sum_{i = 1}^{m} \lambda_i F_i + \indicator_S} (x)$ is continuous at every~$(x, \lambda) \in S \times \Delta^m$ from~\cite[Theorem~7.37]{Rockafellar1998},~$h$ is also continuous on~$S \times \Delta^m$.
    Moreover, \cref{thm: smoothness of Moreau envelope} implies that for all~$x, z \in S$ the function~$h_\lambda(\cdot) \coloneqq h(\cdot, \lambda)$ has a directional derivative:
    \[
        h_\lambda^\prime(x; z - x) = \sum_{i = 1}^{m} \lambda_i F_i^\prime(x; z - x) - \ell \left(x - \prox_{\frac{1}{\ell} \sum_{i = 1}^{m} \lambda_i F_i + \indicator_S}(x)\right)^\T (z - x) 
    .\]
    As~$\prox_{ \frac{1}{\ell} \sum_{i = 1}^{m} \lambda_i F_i + \indicator_S } (x)$ is continuous at every~$(x, \lambda) \in S \times \Delta^m$ (cf.~\cite[Exercise~7.38]{Rockafellar1998}),~$h_\lambda^\prime(x; z - x)$ is also continuous at every~$(x, z, \lambda) \in S \times S \times \Delta^m$.
    The discussion above and the compactness of~$\Delta^m$ show that all assumptions of \cref{prop: directional differentiability} are satisfied.
    So, we get the desired result.
\end{proof}

From \cref{thm: u ell,thm: directional differentiability of u_ell}, the weakly Pareto optimal solutions for~\cref{eq: CMOP} are the global optimal solutions of the following directionally differentiable single-objective optimization problem:
\begin{equation} \label{eq: min u_ell}
    \min_{x \in S} \quad u_\ell(x)
.\end{equation} 
Since~$u_\ell$ is generally non-convex,~\cref{eq: min u_ell} may have local optimal solutions or stationary points that are not globally optimal.
As the following example shows, such stationary points are not necessarily Pareto stationary for~\cref{eq: CMOP}.
\begin{example} \label{exm: min u_ell}
    Let~$m = 1, \ell = 1, S = \setR$ and~$F_1(x) = \abs{x}$.
    Then, we have
    \[
        \envelope_{F_1}(x) = \begin{dcases}
            x^2 / 2, & \text{if } \abs{x} < 1, \\
            \abs{x} - 1 / 2, & \otherwise
        .\end{dcases}
    \]
    Hence, we can evaluate~$u_1$ as follows:
    \[
        u_1(x) = \begin{dcases}
            \abs{x} - x^2 / 2, & \text{if } \abs{x} < 1, \\
            1 / 2, & \otherwise
        .\end{dcases}
    \] 
    It is stationary for~\cref{eq: min u_ell} at~$\abs{x} \ge 1$ and~$x = 0$ but minimal only at~$x = 0$.
    Furthermore, the stationary point of~$F_1$ is only~$x = 0$.
\end{example}
However, if we assume the strict convexity of each~$F_i$, then the stationary point of~\cref{eq: min u_ell} is Pareto optimal for~\cref{eq: CMOP} and hence global optimal for~\cref{eq: min u_ell}.
Note that this assumption does not assert the convexity of~$u_\ell$.
\begin{theorem} \label{thm: stationary of u_ell implies weak Pareto optimality}
    Suppose that~$F_i$ is strictly convex for all~$i \in \set*{1, \dots, m}$.
    If~$x \in S$ is a stationary point of~\cref{eq: min u_ell}, i.e.,
    \[
        u_\ell^\prime(x; z - x) \ge 0 \forallcondition{z \in S}
    ,\]
    then~$x$ is Pareto optimal for~\cref{eq: CMOP}.
\end{theorem}
\begin{proof}
    Let~$\lambda \in \Lambda(x)$, where~$\Lambda(x)$ is given by~\cref{eq: Lambda}.
    Then, \cref{thm: directional differentiability of u_ell} gives
    \[
        \sum_{i = 1}^{m} \lambda_i F_i^\prime(x; z - x) - \ell \left( x - \prox_{\frac{1}{\ell} \sum_{i = 1}^{m} \lambda_i F_i + \indicator_S}(x) \right)^\T (z - x) \ge 0 \forallcondition{z \in S}
    .\] 
    Substituting~$z = \prox_{\frac{1}{\ell} \sum_{i = 1}^{m} \lambda_i F_i + \indicator_S}(x)$ into the above inequality, we get
    \[
        \sum_{i = 1}^{m} \lambda_i F_i^\prime \left(x; \prox_{\frac{1}{\ell} \sum_{i = 1}^{m} \lambda_i F_i + \indicator_S}(x) - x \right) + \ell \norm*{ x - \prox_{\frac{1}{\ell} \sum_{i = 1}^{m} \lambda_i F_i + \indicator_S}(x) }^2 \ge 0
    .\]
    On the other hand, \cref{cor: second prox} yields
    \[
        \norm*{x - \prox_{\frac{1}{\ell} \sum_{i = 1}^{m} \lambda_i F_i + \indicator_S}(x)}^2 \le \frac{1}{\ell} \sum_{i = 1}^{m} \lambda_i \left\{ F_i(x) - F_i(\prox_{\frac{1}{\ell} \sum_{i = 1}^{m} \lambda_i F_i + \indicator_S}(x)) \right\} 
    .\]
    Combining the above two inequalities, we have
    \begin{multline*}
        \sum_{i = 1}^{m} \lambda_i F_i^\prime \left(x; \prox_{\frac{1}{\ell} \sum_{i = 1}^{m} \lambda_i F_i + \indicator_S}(x) - x \right) \\
        \ge \sum_{i = 1}^{m} \lambda_i \left\{ F_i(\prox_{\frac{1}{\ell} \sum_{i = 1}^{m} \lambda_i F_i + \indicator_S}(x)) - F_i(x) \right\} 
    .\end{multline*}
    Since~$F_i$ is strictly convex for all~$i \in \set*{1, \dots, m}$, the above inequality implies that~$x = \prox_{\frac{1}{\ell} \sum_{i = 1}^{m} \lambda_i F_i + \indicator_S}(x)$, and hence~$u_\ell(x) = 0$.
    This means that~$x$ is Pareto optimal for~\cref{eq: CMOP} from the strict convexity of~$F_i$, \cref{lem: Pareto relationship}~\ref{enum: weak stationary},~\ref{enum: stationary pareto}, and \cref{thm: u ell}.
\end{proof}

\subsection{A regularized and partially linearized merit function for composite multiobjective optimization} \label{sec: merit composite}
Now, let us consider the composite case, i.e., each component~$F_i$ of the objective function~$F$ of~\eqref{eq: CMOP} has the following structure:
\begin{equation} \label{eq: f + g}
    F_i(x) \coloneqq f_i(x) + g_i(x), \quad i = 1, \dots, m
,\end{equation}
where~$f_i \colon S \to \setR$ is continuously differentiable but not necessarily convex, and~$g_i \colon S \to \setR$ is convex but not necessarily differentiable.
Note that since~$g_i$ is finite and convex, there exists a directional derivative~$g_i'(x; z - x)$ for any~$x, z \in S$.
Such composite objective functions have many applications, particularly in machine learning.
Since they are generally non-convex, we can regard them as a relaxation of the assumptions of the previous subsection.
For~\cref{eq: CMOP} with objective function~\cref{eq: f + g}, we propose a regularized and partially linearized merit function~$w_\ell \colon S \to \setR$ with a given~$\ell > 0$ as follows:
\begin{equation} \label{eq: w_ell}
    w_\ell(x) \coloneqq \max_{y \in S} \min_{i = 1, \dots, m} \left\{ \nabla f_i(x)^\T (x - y) + g_i(x) - g_i(y) - \frac{\ell}{2} \norm*{x - y}^2 \right\} 
.\end{equation} 
Like~$u_\ell$, the convexity of~$g_i$ leads to the finiteness of~$w_\ell$ and the existence of a unique solution that attains~$\max_{y \in S}$.
As the following remark shows,~$w_\ell$ generalizes other kinds of merit functions.
\begin{remark} \label{rem: composite ref}
\begin{enumerate}
    \item When~$g_i = 0$,~$w_\ell$ corresponds to the regularized gap function~\cite{Charitha2010} for vector variational inequality.
    \item When~$f_i = 0$,~$w_\ell$ matches~$u_\ell$ defined by~\cref{eq: u_ell}. \label{enum: w u correspond}
    \end{enumerate}
\end{remark}
As shown in the following theorem,~$w_\ell$ is a merit function in the sense of Pareto stationarity.
\begin{theorem} \label{thm: w ell}
    Let~$w_\ell$ be given by~\cref{eq: w_ell} for some~$\ell > 0$.
    Then, we have~$w_\ell(x) \ge 0$ for all~$x \in S$.
    Furthermore,~$x \in S$ is Pareto stationary for~\cref{eq: CMOP} if and only if~$w_\ell(x) = 0$.
\end{theorem}
\begin{proof}
    We first show the nonnegativity of~$w_\ell$.
    Let~$x \in S$.
    The definition of~$w_\ell$ gives
    \[
        \begin{split}
            w_\ell(x) &= \sup_{y \in S} \min_{i = 1, \dots, m} \left\{ \nabla f_i(x)^\T (x - y) + g_i(x) - g_i(y) - \frac{\ell}{2} \norm*{x - y}^2 \right\} \\
                      &\ge \min_{i = 1, \dots, m} \left\{ \nabla f_i(x)^\T (x - x) + g_i(x) - g_i(x) - \frac{\ell}{2} \norm*{x - x}^2 \right\} = 0
        .\end{split}
    \]

    Let us prove the second statement.
    Assume that~$w_\ell(x) = 0$.
    Then, again using the definition of~$w_\ell$, we get
    \[
        \min_{i = 1, \dots, m} \left\{ \nabla f_i(x)^\T (x - y) + g_i(x) - g_i(y) - \frac{\ell}{2} \norm*{x - y}^2 \right\} \le 0 \forallcondition{y \in S}
    .\] 
    Let~$z \in S$ and~$\alpha \in (0, 1)$.
    Since~$S \subseteq \setR^n$ is convex,~$x, z \in S$ implies~$x + \alpha (z - x) \in S$.
    Therefore, by substituting~$y = x + \alpha (z - x)$ into the above inequality, we obtain
    \[
        \min_{i = 1, \dots, m} \left\{ - \nabla f_i(x)^\T (\alpha (z - x)) + g_i(x) - g_i(x + \alpha (z - x)) - \frac{\ell}{2} \norm*{\alpha (z - x)}^2 \right\} \le 0
    .\] 
    Dividing both sides by~$\alpha$ yields
    \[
        \min_{i = 1, \dots, m} \left\{ - \nabla f_i(x)^\T (z - x) - \frac{g_i(x + \alpha (z - x)) - g_i(x)}{\alpha} - \frac{\ell \alpha}{2} \norm*{z - x}^2 \right\} \le 0
    .\] 
    By taking~$\alpha \searrow 0$ and multiplying both sides by~$- 1$, we get
    \[
        \max_{i = 1, \dots, m} F_i^\prime(x; z - x) \ge 0
    ,\] 
    which means that~$x$ is Pareto stationary for~\cref{eq: CMOP}.

    Now, we prove the converse by indirect proof.
    Suppose that~$w_\ell(x) > 0$.
    Then, from the definition of~$w_\ell$, there exists some~$y \in S$ such that
    \[
        \min_{i = 1, \dots, m} \left\{ \nabla f_i(x)^\T (x - y) + g_i(x) - g_i(y) - \frac{\ell}{2} \norm*{x - y}^2 \right\} > 0
    .\] 
    Since~$g_i$ is convex, we obtain
    \[
        \min_{i = 1, \dots, m} \left\{ \nabla f_i(x)^\T (x - y) - g_i^\prime(x; y - x) - \frac{\ell}{2} \norm*{x - y}^2 \right\} > 0
    .\] 
    Thus, we have
    \[
        \max_{i = 1, \dots, m} F_i^\prime(x; y - x) \le - \frac{\ell}{2} \norm*{x - y}^2 < 0
    ,\] 
    which shows that~$x$ is not Pareto stationary for~\cref{eq: CMOP}.
\end{proof}
While~$u_0$ and~$u_\ell$ given by~\cref{eq: u_0,eq: u_ell} are merit functions in the sense of weak Pareto optimality,~$w_\ell$ defined by~\cref{eq: w_ell} is a merit function only in the sense of Pareto stationarity.
As indicated by the following example, even if~$w_\ell(x) = 0$,~$x$ is not necessarily weakly Pareto optimal for~\cref{eq: CMOP}.
\begin{example}
    Consider the single-objective function~$F\colon \setR \to \setR$ defined by~$F(x) \coloneqq f(x) + g(x)$, where
    \[
        f(x) \coloneqq - x^2 \eqand g(x) \coloneqq 0
    ,\]
    and set~$S = \setR$.
    Then, we have
    \[
        w_\ell(0) = \max_{y \in \setR} \left\{ f^\prime(0) (0 - y) + g(0) - g(y) - \frac{\ell}{2} (y - 0)^2 \right\} 
                  = \max_{y \in \setR} \left\{ - \frac{\ell}{2} y^2 \right\} = 0
    ,\] 
    but $x = 0$ is not global minimal (i.e., weakly Pareto optimal) for~$F$.
\end{example}

We now define the optimal solution mapping~~$W_\ell \colon S \to S$ associated with~\cref{eq: w_ell} by
\begin{equation} \label{eq: W_ell}
    W_\ell(x) \coloneqq \argmax_{y \in S} \min_{i = 1, \dots, m} \left\{ \nabla f_i(x)^\T (x - y) + g_i(x) - g_i(y) - \frac{\ell}{2} \norm*{x - y}^2 \right\} 
.\end{equation} 
From the optimality condition of the maximization problem associated with~\cref{eq: w_ell,eq: W_ell} and~\cite[Proposition~A.22]{Bertsekas1971}, we obtain
\[
    \ell (x - W_\ell(x)) \in \conv_{i \in \mathcal{I}(x)} \left[ \nabla f_i(x) + \partial g_i(W_\ell(x)) \right] + N_S(W_\ell(x)) \forallcondition{x \in S}
,\]
where~$N_S$ is the normal cone~\cref{eq: normal cone} to the convex set~$S$ and
\[
    \mathcal{I}(x) \coloneqq \argmin_{i = 1, \dots, m} \left[ \nabla f_i(x)^\T (x - W_\ell(x)) + g_i(x) - g_i(W_\ell(x)) \right]
.\]
Therefore, from~\cref{eq: normal cone,eq: subdifferential}, for any~$x \in S$ there exists~$\lambda (x)$ beloging to the unit~$m$-simplex~$\Delta^m$ defined by~\cref{eq: simplex} such that~$\lambda_j(x) = 0$ for all~$j \notin \mathcal{I}(x)$ and
\begin{equation} \label{eq:reg_lin_gap_optimality}
    \ell (x - W_\ell(x))^\T (z - W_\ell(x)) \le \sum_{i = 1}^m \lambda_i(x) \left[\nabla f_i(x)^\T (z - W_\ell(x)) + g_i(z) - g_i(W_\ell(x))\right]
\end{equation}
for all~$z \in S$.
Particularly, if we substitute~$z = x$, we get
\[
    \ell \norm{x - W_\ell(x)}^2 \le w_\alpha(x) + \frac{\ell}{2} \norm{x - W_\ell(x)}^2
,\]
which reduces to
\begin{equation} \label{eq:reg_lin_gap_MO_LB}
    w_\ell(x) \ge \frac{\ell}{2} \norm{x - W_\ell(x)}^2
.\end{equation}

We can also show the continuity of~$w_\ell$ and~$W_\ell$.
\begin{theorem}
    For all~$\ell > 0$,~$w_\ell$ and~$W_\ell$ defined by~\cref{eq: w_ell,eq: W_ell} are continuous on~$S$.
    Moreover, if every~$\nabla f_i$ is locally Lipschitz continuous for~$i = 1, \dots, m$,~$w_\ell$ and~$W_\ell$ are locally Lipschitz continuous and locally H\"older continuous with exponent~$1 / 2$, respectively, on~$S$.
\end{theorem}
\begin{proof}
    Let~$\Omega$ be a bounded subset of~$S$ and let~$x^1, x^2 \in \Omega$.
    Adding the two inequalities gotten by substituting~$(x, z) = (x^1, W_\ell(x^2))$ and~$(x, z) = (x^2, W_\ell(x^1))$ into~\cref{eq:reg_lin_gap_optimality}, we obtain
    \begin{align*}
        \MoveEqLeft \frac{1}{\ell} \left( W_\ell(x^1) - W_\ell(x^2) - \left(x^1 - x^2\right) \right)^\T \left( W_\ell(x^1) - W_\ell(x^2) \right) \\
        \le{} & \sum_{i = 1}^{m} \lambda_i(x^1)\left[ \nabla f_i(x^1)^\T (x^1 - W_\ell(x^1)) + g_i(x^1) - g_i\left(W_\ell(x^1)\right) \right]    \\
              & + \sum_{i = 1}^{m} \lambda_i(x^2) \left[ \nabla f_i(x^2)^\T (x^2 - W_\ell(x^2)) + g_i(x^2) - g_i\left(W_\ell(x^2)\right) \right] \\
              & + \sum_{i = 1}^{m} \lambda_i(x^1) \left[ \nabla f_i(x^1)^\T (W_\ell(x^2) - x^1) + g_i\left(W_\ell(x^2)\right) - g_i(x^1) \right] \\
              & + \sum_{i = 1}^{m} \lambda_i(x^2) \left[ \nabla f_i(x^2)^\T (W_\ell(x^1) - x^2) + g_i\left(W_\ell(x^1)\right) - g_i(x^2) \right]
        .\end{align*}
    Since~$\lambda_j(x) \neq 0$ for~$j \in \mathcal{I}(x)$, we have
    \begin{align*}
        \MoveEqLeft \frac{1}{\ell} \left( W_\ell(x^1) - W_\ell(x^2) - \left(x^1 - x^2\right) \right)^\T (W_\ell(x^1) - W_\ell(x^2)) \\
        \le{} & \min_{i = 1, \dots, m} \left[ \nabla f_i(x^1)^\T (x^1 - W_\ell(x^1)) + g_i(x^1) - g_i\left(W_\ell(x^1)\right) \right]                       \\
              & + \min_{i = 1, \dots, m} \left[ \nabla f_i(x^2)^\T (x^2 - W_\ell(x^2)) + g_i(x^2) - g_i\left(W_\ell(x^2)\right) \right]                     \\
              & + \sum_{i = 1}^{m} \lambda_i(x^1) \left[ \nabla f_i(x^1)^\T (W_\ell(x^2) - x^1) + g_i\left(W_\ell(x^2)\right) - g_i(x^1) \right] \\
              & + \sum_{i = 1}^{m} \lambda_i(x^2) \left[ \nabla f_i(x^2)^\T (W_\ell(x^1) - x^2) + g_i\left(W_\ell(x^1)\right) - g_i(x^2) \right] \\
        \le{} & \sum_{i = 1}^{m} \lambda_i(x^2) \left[ \nabla f_i(x^1)^\T (x^1 - W_\ell(x^1)) + g_i(x^1) - g_i\left(W_\ell(x^1)\right) \right]   \\
              & + \sum_{i = 1}^{m} \lambda_i(x^1) \left[ \nabla f_i(x^2)^\T (x^2 - W_\ell(x^2)) + g_i(x^2) - g_i\left(W_\ell(x^2)\right) \right] \\
              & + \sum_{i = 1}^{m} \lambda_i(x^1) \left[ \nabla f_i(x^1)^\T (W_\ell(x^2) - x^1) + g_i\left(W_\ell(x^2)\right) - g_i(x^1) \right] \\
              & + \sum_{i = 1}^{m} \lambda_i(x^2) \left[ \nabla f_i(x^2)^\T (W_\ell(x^1) - x^2) + g_i\left(W_\ell(x^1)\right) - g_i(x^2) \right]
        .\end{align*}
    Therefore, simple calculations give
    \begin{multline} \label{eq:cont_w_alpha:cont}
        \frac{1}{\ell} \norm*{W_\ell(x^1) - W_\ell(x^2)}^2 \le \frac{1}{\ell} (W_\ell(x^1) - W_\ell(x^2))^\T (x^1 - x^2) \\
        + \sum_{i = 1}^{m} \left[ \lambda(x^2) - \lambda(x^1) \right] \left[ g_i(x^1) - g_i(x^2) + \nabla f_i(x^1)^\T (x^1 - x^2) \right. \\
            \left. - (\nabla f_i(x^1) - \nabla f_i(x^2))^\T x^2 \right] \\
        + \sum_{i = 1}^{m} \lambda_i(x^1) (\nabla f_i(x^1) - \nabla f_i(x^2))^\T W_\ell(x^2)
        + \sum_{i = 1}^{m} \lambda_i(x^2) (\nabla f_i(x^2) - \nabla f_i(x^1))^\T W_\ell(x^1)
        .\end{multline}
    When~$x^1 \to x^2$, the right-hand side tends to zero, which means the continuity of~$W_\ell$ on~$C$.
    Therefore, from the definition, we can also say that~$w_\ell$ is continuous on~$C$ immediately.

    Assume that each~$\nabla f_i, i = 1, \dots, m$ is locally Lipschitz continuous.
    Since~$g_i$ is also locally Lipschitz continuous from~\cref{thm:local_Lipschitz}, we can prove the local H\"older continuity of~$W_\ell$ from~\cref{eq:cont_w_alpha:cont}.
    On the other hand, the definitions~\cref{eq: w_ell,eq: W_ell} of~$w_\ell$ and~$W_\ell$ give
    \begin{align*}
              & w_\ell(x^1) - w_\ell(x^2)                                                                                                                                                                                           \\
        ={}   & \min_{i = 1, \dots, m} \left[ \nabla f_i(x^1)^\T (x^1 - W_\ell(x^1)) + g_i(x^1) - g_i\left(W_\ell(x^1)\right) \right] - \frac{\ell}{2} \norm*{x^1 - W_\ell(x^1)}^2 \\
              & - \max_{y \in S} \min_{i = 1, \dots, m} \left[ \nabla f_i(x^2)^\T (x^2 - y) + g_i(x^2) - g_i(y) - \frac{\ell}{2} \norm*{x^2 - y}^2 \right] \\
        \le{} & \min_{i = 1, \dots, m} \left[ \nabla f_i(x^1)^\T (x^1 - W_\ell(x^1)) + g_i(x^1) - g_i\left(W_\ell(x^1)\right) \right] - \frac{\ell}{2} \norm*{x^1 - W_\ell(x^1)}^2 \\
              & - \min_{i = 1, \dots, m} \left[ \nabla f_i(x^2)^\T (x^2 - W_\ell(x^1)) + g_i(x^2) - g_i\left(W_\ell(x^1)\right) \right] + \frac{\ell}{2} \norm*{x^2 - W_\ell(x^1)}^2 \\
        \le{} & \max_{i = 1, \dots, m} \left[ \left( \nabla f_i(x^1) - \nabla f_i(x^2) \right)^\T \left( x^1 - W_\ell(x^1) \right) + \nabla f_i(x^2)^\T (x^1 - x^2) + g_i(x^1) - g_i(x^2) \right] \\
              & - \frac{\ell}{2} \left( x^1 - x^2 \right)^\T \left( x^1 + x^2 - 2 W_\ell(x^1) \right) \\
        \le{} & \norm*{x^1 - W_\ell(x^1)} \max_{i = 1, \dots, m} \norm*{\nabla f_i(x^1) - \nabla f_i(x^2)} \\
              & + \max_{i = 1, \dots, m} \norm*{\nabla f_i(x^2)} \norm{x^1 - x^2} + \max_{i = 1, \dots, m} \abs{g_i(x^1) - g_i(x^2)} \\
              & + \frac{\ell}{2} \norm*{x^1 + x^2 - 2 W_\ell(x^1)} \norm*{x^1 - x^2}
        ,\end{align*}
    where the first inequality comes from the inequality~$\min_{i = 1, \dots, m} v^1_i - \min_{i = 1, \dots, m} v^2_i \ge \min_{i = 1, \dots, m} (v^1_i - v^2_i)$ for any~$v^1, v^2 \in \setR^m$, and the third inequality follows from the Cauchy-Schwarz inequality.
    The above inequality holds even if we interchange~$x^1$ and~$x^2$.
    Furthermore,~$W_\ell(x)$ and~$\nabla f_i(x)$ are bounded for any~$x \in \Omega$ due to their continuity.
    Therefore, local Lipschitz continuity of~$\nabla f_i$ and~$g_i$ implies the local Lipschitz continuity of~$w_\ell$.
\end{proof}

On the other hand, in the same way as the derivation of~\cref{eq: u dual}, Sion's minimax theorem~\cite{Sion1958} gives another representation of~$w_\ell$ for~$\ell > 0$ as follows:
\begin{equation} \label{eq: w_ell Sion}
    w_\ell(x) = \min_{\gamma \in \Delta^m} \max_{y \in S} \sum_{i = 1}^{m} \gamma_i \left\{ \nabla f_i(x)^\T (x - y) + g_i(x) - g_i(y) - \frac{\ell}{2} \norm*{x - y}^2 \right\} 
,\end{equation} 
where~$\Delta^m$ denotes the standard simplex~\cref{eq: simplex}.
Moreover, simple calculations show that
\begin{equation} \label{eq: w_ell dual}
    \begin{alignedat}{2}
        w_\ell(x) &= \min_{\gamma \in \Delta^m} &&\left[ \sum_{i = 1}^{m} \gamma_i g_i(x) + \frac{1}{2 \ell} \norm*{\sum_{i = 1}^{m} \gamma_i \nabla f_i(x)}^2 \right. \\
                  &&& \left. {}- \min_{y \in S} \left\{ \sum_{i = 1}^{m} \gamma_i g_i(y) + \frac{\ell}{2} \norm*{x - \frac{1}{\ell} \sum_{i = 1}^{m} \gamma_i \nabla f_i(x) - y}^2 \right\} \right] \\
                  &= \min_{\gamma \in \Delta^m} &&\left[ \sum_{i = 1}^{m} \gamma_i g_i(x) + \frac{1}{2 \ell} \norm*{\sum_{i = 1}^{m} \gamma_i \nabla f_i(x)}^2 \right. \\
                  &&& \left. {}- \ell \envelope_{\frac{1}{\ell} \sum_{i = 1}^{m} \gamma_i g_i + \indicator_S} \left( x - \frac{1}{\ell} \sum_{i = 1}^{m} \gamma_i \nabla f_i(x) \right) \right]
              ,\end{alignedat}
          \end{equation} 
where~$\envelope$ and~$\indicator$ is given by~\cref{eq: Moreau envelope,eq: indicator function}, respectively.
In other words, we can compute~$w_\ell$ via the following~$m$-dimensional, simplex-constrained, and convex optimization problem:
\begin{equation} \label{eq: dual of w_ell}
    \begin{aligned}
        \min_{\gamma \in \setR^m} &&& \sum_{i = 1}^{m} \gamma_i g_i(x) + \frac{1}{2 \ell} \norm*{\sum_{i = 1}^{m} \gamma_i \nabla f_i(x)}^2 \\
                                  &&&\quad - \ell \envelope_{\frac{1}{\ell} \sum_{i = 1}^{m} \gamma_i g_i + \indicator_S}\left( x - \frac{1}{\ell} \sum_{i = 1}^{m} \gamma_i \nabla f_i(x) \right)  \\
        \st &&& \gamma \ge 0 \eqand \sum_{i = 1}^{m} \gamma_i = 1 
    .\end{aligned}
\end{equation} 
Moreover, the following theorem proves that the objective function of~\cref{eq: dual of w_ell} is continuously differentiable.
\begin{theorem} \label{thm: differentiability of the objective function of w_ell}
    Let~$x \in S$ be given.
    The objective function of~\cref{eq: dual of w_ell} is continuously differentiable at every~$\gamma \in \setR^m$ and
    \begin{multline*}
        \nabla_\gamma \left[ \sum_{i = 1}^{m} \gamma_i g_i(x) + \frac{1}{2 \ell} \norm*{\sum_{i = 1}^{m} \gamma_i \nabla f_i(x)}^2 - \ell \envelope_{\frac{1}{\ell} \sum_{i = 1}^{m} g_i + \indicator_S}\left( x - \frac{1}{\ell} \sum_{i = 1}^{m} \gamma_i \nabla f_i(x) \right)  \right] \\
        = g(x) - g\left( \prox_{\frac{1}{\ell} \sum_{i = 1}^{m} \gamma_i g_i + \indicator_S}\left( x - \frac{1}{\ell} \sum_{i = 1}^{m} \gamma_i \nabla f_i(x) \right) \right) \\
        - J_f(x) \left( \prox_{\frac{1}{\ell} \sum_{i = 1}^{m} \gamma_i g_i + \indicator_S}\left( x - \frac{1}{\ell} \sum_{i = 1}^{m} \gamma_i \nabla f_i(x) \right) - x \right) 
    ,\end{multline*}
    where~$\prox$ is the proximal operator~\cref{eq: proximal operator}, and~$J_f(x)$ is the Jacobian matrix at~$x$ given by
    \[
        J_f(x) \coloneqq \left( \nabla f_1(x), \dots, \nabla f_m(x) \right)^\T
    .\]
\end{theorem}
\begin{proof}
    Let
    \[
        \theta(y, \lambda) \coloneqq \sum_{i = 1}^{m} \lambda_i g_i(y) + \frac{\ell}{2} \norm*{x - \frac{1}{\ell} \sum_{i = 1}^{m} \gamma_i \nabla f_i(x) - y}^2
    .\] 
    Then,~$\theta$ is continuous,~$\theta_y(\cdot) \coloneqq \theta(y, \cdot)$ is continuously differentiable, and
    \[
        \nabla_\gamma \theta_y(\gamma) = g(y) + J_f(x) \left( y - x + \frac{1}{\ell} \sum_{i = 1}^{m} \gamma_i \nabla f_i(x) \right) 
    .\] 
    Moreover,~$\prox_{\frac{1}{\ell} \sum_{i = 1}^{m} \gamma_i g_i + \indicator_S}(x) = \argmin_{y \in S} \theta(y, \lambda)$ is also continuous at every~$\gamma \in \setR^m$ (cf.~\cite[Excercise 7.38]{Rockafellar1998}).
    The above discussion implies that every assumption in \cref{prop: directional differentiability} is satisfied, as well as the proof of \cref{thm: differentiability of the objective function of u_ell}.
    Combined with the uniqueness of~$\prox_{\frac{1}{\ell} \sum_{i = 1}^{m} \gamma_i g_i + \indicator_S}(x)$, we get
    \begin{multline*}
        \nabla_\gamma \left[ \ell \envelope_{\frac{1}{\ell} \sum_{i = 1}^{m} \gamma_i g_i + \indicator_S}\left( x - \frac{1}{\ell} \sum_{i = 1}^{m} \gamma_i \nabla f_i(x) \right) \right] \\ 
    = g\left( \prox_{\frac{1}{\ell} \sum_{i = 1}^{m} \gamma_i g_i + \indicator_S}\left( x - \frac{1}{\ell} \sum_{i = 1}^{m} \gamma_i \nabla f_i(x) \right) \right) \\
    + J_f(x) \left( \prox_{\frac{1}{\ell} \sum_{i = 1}^{m} \gamma_i g_i + \indicator_S}\left( x - \frac{1}{\ell} \sum_{i = 1}^{m} \gamma_i \nabla f_i(x) \right) - x + \frac{1}{\ell} \sum_{i = 1}^{m} \gamma_i \nabla f_i(x) \right) 
    .\end{multline*}
    On the other hand, we have
    \[
        \nabla_\gamma \left[ \sum_{i = 1}^{m} \gamma_i g_i(x) + \frac{1}{2 \ell} \norm*{\sum_{i = 1}^{m} \gamma_i \nabla f_i(x)}^2 \right] = g(x) + \frac{1}{\ell} J_f(x) \sum_{i = 1}^{m} \gamma_i \nabla f_i(x) 
    .\] 
    Adding the above two equalities, we obtain the desired result.
\end{proof}
Thus, like~\cref{eq: dual of u_ell}, \cref{eq: dual of w_ell} is solvable with convex optimization techniques such as the interior point method~\cite{Bertsekas1999} when we can quickly evaluate~$\prox_{\frac{1}{\ell} \sum_{i = 1}^{m} \lambda_i g_i + \indicator_S}(\cdot)$.
When~$n \gg m$, this usually gives a faster way to compute~$w_\ell$.
Note, for example, that if~$g_i(x) = 0$ for all~$i = 1, \dots, m$, then~$\prox_{\frac{1}{\ell} \sum_{i = 1}^{m} \lambda_i g_i + \indicator_S}$ reduces to the projection onto~$S$ from~\cref{eq: proximal operator of the indicator function}.
Moreover, for example, if~$g_i(x) = g_1(x)$ for any~$i = 1, \dots, m$, or if~$g_i(x) = g_1(x_{I_i})$ and the index sets~$I_i \subseteq \set*{1, \dots, n}$ do not overlap each other, then~$\prox_{\frac{1}{\ell} \sum_{i = 1}^{m} \lambda_i g_i}$ is computable with each~$\prox_{g_i}$ when~$S = \setR^n$.

Now, define the optimal solution set of~\cref{eq: dual of w_ell} by
\begin{multline} \label{eq: Gamma}
    \Gamma(x) = \argmin_{\gamma \in \Delta^m} \left\{ \sum_{i = 1}^{m} \gamma_i g_i(x) + \frac{1}{2 \ell} \norm*{\sum_{i = 1}^{m} \gamma_i \nabla f_i(x)}^2 \right. \\
    \left. {}- \ell \envelope_{\frac{1}{\ell} \sum_{i = 1}^{m} \gamma_i g_i + \indicator_S}\left( x - \frac{1}{\ell} \sum_{i = 1}^{m} \gamma_i \nabla f_i(x) \right) \right\} 
.\end{multline} 
Then, in the same manner as \cref{thm: directional differentiability of u_ell}, we obtain the following theorem.
\begin{theorem} \label{thm: directional differentiability of w_ell}
    Let~$x \in S$.
    Assume that~$f_i$ is twice continuously differentiable at~$x$.
    Then, for all~$\ell > 0$, the merit function~$w_\ell$ defined by~\cref{eq: w_ell} has a directional derivative
    \begin{multline*}
        w_\ell^\prime(x; z - x) = \inf_{\gamma \in \Gamma(x)} \left\{ \sum_{i = 1}^{m} \gamma_i g_i^\prime(x; z - x) \right. \\
        {}- \ell \left( \left[ I - \frac{1}{\ell} \sum_{i = 1}^{m} \gamma_i \nabla^2 f_i(x) \right] \left[x - \prox_{\frac{1}{\ell} \sum_{i = 1}^{m} \gamma_i g_i + \indicator_S} \left( x - \frac{1}{\ell} \sum_{i = 1}^{m} \gamma_i \nabla f_i(x) \right) \right] \right.\\
        \left. \left.- \frac{1}{\ell} \sum_{i = 1}^{m} \gamma_i \nabla f_i(x) \right)^\T (z - x) \right\} 
    \end{multline*}
    for all~$z \in S$, where~$\prox$ and~$\Gamma$ is given by~\cref{eq: proximal operator,eq: Gamma}, respectively, and~$I$ is the~$n$-dimensional identity matrix.
    In particular, if~$\Gamma(x)$ is a singleton, i.e.,~$\Gamma(x) = \set*{\gamma(x)}$, and~$g_i$ is continuously differentiable at~$x$, then~$w_\ell$ is continuously differentiable at~$x$, and we have
    \begin{multline*}
        \nabla w_\ell(x) = \sum_{i = 1}^{m} \gamma_i(x) \nabla F_i(x) \\
        {}- \ell \left[ I - \frac{1}{\ell} \sum_{i = 1}^{m} \gamma_i(x) \nabla^2 f_i(x) \right] \left[x - \prox_{\frac{1}{\ell} \sum_{i = 1}^{m} \gamma_i(x) g_i + \indicator_S} \left( x - \frac{1}{\ell} \sum_{i = 1}^{m} \gamma_i(x) \nabla f_i(x) \right) \right]
    .\end{multline*}
\end{theorem}
If the convex part~$g_i$ is the same regardless of~$i$, we get the following corollary without assuming the differentiability of~$g_i$.
\begin{corollary} \label{cor: w_ell composite}
    Let~$x \in S$ and~$\ell > 0$.
    Assume that~$f_i$ is twice continuously differentiable at~$x$ and~$g_i = g_1$ for all~$i = 1, \dots, m$, and recall that~$w_\ell$ and~$\prox$ be defined by~\cref{eq: w_ell,eq: proximal operator}, respectively.
    If~$\Gamma(x)$ given by~\cref{eq: Gamma} is a singleton, i.e.,~$\Gamma(x) = \set*{\gamma(x)}$, then the function~$w_\ell - g_1$ is continuously differentiable at~$x$, and we have
    \begin{multline*}
        \nabla_x \left( w_\ell(x) - g_1(x) \right) \\
        = - \ell \left[ I - \frac{1}{\ell} \sum_{i = 1}^{m} \gamma_i(x) \nabla^2 f_i(x) \right] \left[x - \prox_{\frac{1}{\ell} g_1 + \indicator_S} \left( x - \frac{1}{\ell} \sum_{i = 1}^{m} \gamma_i(x) \nabla f_i(x) \right) \right] \\
        + \sum_{i = 1}^{m} \gamma_i(x) \nabla f_i(x)
    .\end{multline*}
\end{corollary}
\Cref{cor: w_ell composite} implies that, under certain conditions, the merit function~$w_\ell = (w_\ell - g_1) + g_1$ is composite, i.e., the sum of a continuously differentiable function and a convex one.

\Cref{thm: w ell,thm: directional differentiability of w_ell} show that the Pareto stationary points for~\cref{eq: CMOP} are global optimal for the following directionally differentiable single-objective optimization problem:
\begin{equation} \label{eq: min w_ell}
    \min_{x \in S} \quad w_\ell(x)
.\end{equation} 
Moreover, when the assumptions of \cref{cor: w_ell composite} hold, we can apply first-order methods such as the proximal gradient method~\cite{Fukushima1981} to~\cref{eq: min w_ell}.
On the other hand, if we consider~\cref{exm: min u_ell} with~$f_i = 0$, we can see that the stationary point for~\cref{eq: min w_ell} is not necessarily Pareto stationary for~\cref{eq: CMOP}.
However, if~$f_i$ is convex and twice continuously differentiable, and~$F_i$ is strictly convex, then we can prove that every stationary point of~\cref{eq: min w_ell} is Pareto optimal for~\cref{eq: CMOP}, i.e., global optimal for~\cref{eq: min u_ell}.
Note that this assumption does not assert the convexity of~$w_\ell$.
\begin{theorem}
    Let~$x \in S$ and~$\ell > 0$.
    Suppose that~$f_i$ is convex and twice continuously differentiable at~$x$, and~$F_i$ is strictly convex for any~$i = 1, \dots, m$.
    If~$x$ is stationary for~\cref{eq: min w_ell}, i.e.,
    \[
        w_\ell^\prime(x; z - x) \ge 0 \forallcondition{z \in S}
    ,\] 
    then~$x$ is Pareto optimal for~\cref{eq: CMOP}.
\end{theorem}
\begin{proof}
    Let~$z \in S$ and $\gamma \in \Gamma(x)$, where~$\Gamma(x)$ is defined by~\cref{eq: Gamma}.
    Then, it follows from~\cref{thm: directional differentiability of w_ell} that
    \begin{multline*}
    \sum_{i = 1}^{m} \gamma_i g_i^\prime(x; z - x) \\
    - \ell \left( \left[ I - \frac{1}{\ell} \sum_{i = 1}^{m} \gamma_i \nabla^2 f_i(x) \right] \left[x - \prox_{\frac{1}{\ell} \sum_{i = 1}^{m} \gamma_i g_i + \indicator_S} \left( x - \frac{1}{\ell} \sum_{i = 1}^{m} \gamma_i \nabla f_i(x) \right) \right] \right.\\
     \left.- \frac{1}{\ell} \sum_{i = 1}^{m} \gamma_i \nabla f_i(x) \right)^\T (z - x) \ge 0
    .\end{multline*}
    Substituting~$z = \prox_{\frac{1}{\ell} \sum_{i = 1}^{m} \lambda_i F_i + \indicator_S}(x)$, we have
    \begin{multline*}
        \sum_{i = 1}^{m} \gamma_i F_i^\prime \left(x; \prox_{\frac{1}{\ell} \sum_{i = 1}^{m} \gamma_i g_i + \indicator_S}\left( x - \frac{1}{\ell} \sum_{i = 1}^{m} \gamma_i \nabla f_i(x) \right) - x \right) \\
        + \ell \left[x - \prox_{\frac{1}{\ell} \sum_{i = 1}^{m} \gamma_i g_i + \indicator_S}\left( x - \frac{1}{\ell} \sum_{i = 1}^{m} \gamma_i \nabla f_i(x) \right) \right]^\T \left[ I - \frac{1}{\ell} \sum_{i = 1}^{m} \gamma_i \nabla^2 f_i(x) \right] \\
        \left[ x - \prox_{\frac{1}{\ell} \sum_{i = 1}^{m} \gamma_i g_i + \indicator_S}\left( x - \frac{1}{\ell} \sum_{i = 1}^{m} \gamma_i \nabla f_i(x) \right) \right] \ge 0
    .\end{multline*}
    Since the convexity of~$f_i$ implies that~$\nabla^2 f_i(x)$ is positive semidefinite, we get
    \begin{multline*}
        \sum_{i = 1}^{m} \gamma_i F_i^\prime \left(x; \prox_{\frac{1}{\ell} \sum_{i = 1}^{m} \gamma_i g_i + \indicator_S}\left( x - \frac{1}{\ell} \sum_{i = 1}^{m} \gamma_i \nabla f_i(x) \right) - x \right) \\
    + \ell \norm*{\prox_{\frac{1}{\ell} \sum_{i = 1}^{m} \gamma_i g_i + \indicator_S}\left( x - \frac{1}{\ell} \sum_{i = 1}^{m} \gamma_i \nabla f_i(x) \right)}^2 \ge 0
    .\end{multline*}
    Therefore, with similar arguments used in the proof of~\cref{thm: stationary of u_ell implies weak Pareto optimality}, we obtain~$x = \prox_{\frac{1}{\ell} \sum_{i = 1}^{m} \gamma_i g_i + \indicator_S}\left( x - (1 / \ell) \sum_{i = 1}^{m} \gamma_i \nabla f_i(x) \right)$, and thus~$w_\ell(x) = 0$.
    Since~$F_i$ is strictly convex,~$x$ is Pareto optimal for~\cref{eq: CMOP} from \cref{lem: Pareto relationship}~\ref{enum: stationary pareto} and \cref{thm: w ell}.
\end{proof}

\section{Relation between different merit functions}
This section assumes that the problem has a composite structure~\cref{eq: f + g} and discusses the connection between the merit functions proposed in \cref{sec: merit continuous,sec: merit convex,sec: merit composite}.
First, we show some inequalities between different types of merit functions.
\begin{theorem} \label{thm: merit between}
    Let~$u_0$, $u_\ell$, and $w_\ell$ be defined by~\cref{eq: u_0,eq: u_ell,eq: w_ell}, respectively, for all~$\ell > 0$.
    Then, the following statements hold.
    \begin{enumerate}
        \item If~$f_i$ is~$\mu_i$-convex for some~$\mu_i \in \setR$ and~$\mu = \min_{i = 1, \dots, m} \mu_i$, then we have
            \[
                \begin{dcases}
                    u_0(x) \le w_\mu(x) \eqand u_\ell(x) \le w_{\mu + \ell}(x), & \text{if } \mu \ge 0, \\
                    u_{- \mu + \ell}(x) \le w_\ell(x), & \text{otherwise}
                \end{dcases}
            \]
            for all~$\ell > 0$ and~$x \in S$. \label{enum: merit between convex}

        \item If~$\nabla f_i$ is $L_i$-Lipschitz continuous for some~$L_i > 0$ and~$L = \max_{i = 1, \dots, m} L_i$, then we get
            \[
                u_{L + \ell}(x) \le w_\ell(x), \quad u_0(x) \ge w_L(x), \eqand u_\ell(x) \ge w_{L + \ell}(x)
            \]
            for all~$\ell > 0$ and~$x \in S$.
            \label{enum: merit between Lipschitz}
    \end{enumerate}
\end{theorem}
\begin{proof}
\ref{enum: merit between convex}:
Let~$i \in \set*{1, \dots, m}$.
The $\mu_i$-convexity of~$f_i$ gives
\[
    f_i(x) - f_i(y) \le \nabla f_i(x)^\T (x - y) - \frac{\mu_i}{2} \norm*{x - y}^2
.\]
By the definition of~$\mu$, we get
\[
    f_i(x) - f_i(y) \le \nabla f_i(x)^\T (x - y) - \frac{\mu}{2} \norm*{x - y}^2
.\]
Thus, recalling~\cref{eq: f + g}, we have
\begin{align*}
    &F_i(x) - F_i(y) \le \nabla f_i(x)^\T (x - y) + g_i(x) - g_i(y) - \frac{\mu}{2} \norm{x - y}^2, \\
    &F_i(x) - F_i(y) - \frac{\ell}{2}\norm*{x - y}^2 \le \nabla f_i(x)^\T (x - y) + g_i(x) - g_i(y) - \frac{\mu + \ell}{2} \norm{x - y}^2, \\
    &F_i(x) - F_i(y) - \frac{- \mu + \ell}{2}\norm*{x - y}^2 \le \nabla f_i(x)^\T (x - y) + g_i(x) - g_i(y) - \frac{\ell}{2} \norm{x - y}^2
,\end{align*}
so the desired inequalities are clear from~\cref{eq: u_0,eq: u_ell,eq: w_ell}.

\ref{enum: merit between Lipschitz}:
Let~$i \in \set*{1, \dots, m}$.
Suppose that~$\nabla f_i$ is $L_i$-Lipschitz continuous.
Then, the descent lemma~\cite[Proposition A.24]{Bertsekas1999} yields
\[
    \abs*{f_i(y) - f_i(x) - \nabla f_i(x)^\T (y - x)} \le \frac{L_i}{2}\norm*{x - y}^2.
\]
By the definition of~$L$, we have
\[
    \abs*{f_i(y) - f_i(x) - \nabla f_i(x)^\T (y - x)} \le \frac{L}{2}\norm*{x - y}^2.
\]
This gives
\begin{align*}
    &F_i(x) - F_i(y) - \frac{L + \ell}{2} \norm*{x - y}^2 \le \nabla f_i(x)^\T (x - y) + g_i(x) - g_i(y) - \frac{\ell}{2} \norm*{x - y}^2, \\
    &F_i(x) - F_i(y) \ge \nabla f_i(x)^\T (x - y) + g_i(x) - g_i(y) - \frac{L}{2} \norm*{x - y}^2, \\
    &F_i(x) - F_i(y) - \frac{\ell}{2} \norm*{x - y}^2 \ge \nabla f_i(x)^\T (x - y) + g_i(x) - g_i(y) - \frac{L + \ell}{2} \norm*{x - y}^2
.\end{align*}
Therefore, we immediately get~$u_{L + \ell}(x) \le w_\ell(x)$,~$u_0(x) \ge w_L(x)$, and~$u_\ell(x) \ge w_{L + \ell}(x)$ for all~$x \in S$ by~\cref{eq: u_0,eq: u_ell,eq: w_ell}.
\end{proof}

Second, we present the relation between coefficients and the proposed merit functions' values.
\begin{theorem} \label{thm: merit inner}
    Recall that~$w_\ell$ is defined by~\cref{eq: w_ell} for all~$\ell > 0$.
    Let~$r$ be an arbitrary scalar such that $r \ge \ell$.
    Then, we get
    \[
        w_r(x) \le w_\ell(x) \le \frac{r}{\ell} w_r(x) \forallcondition{x \in S}
    .\]
\end{theorem}
\begin{proof}
    Let~$x \in S$.
    Since~$r \ge \ell > 0$, the definition~\cref{eq: w_ell} of~$w_r$ and~$w_\ell$ clearly gives the first inequality.
    Thus, we prove the second one.
    From the definition~\eqref{eq: w_ell} of $w_\ell$, we have
\begin{align*}
    \MoveEqLeft w_\ell(x) = \sup_{y \in S} \min_{i = 1, \dots, m} \left\{ \nabla f_i(x)^\T (x - y) + g_i(x) - g_i(y) - \frac{\ell}{2}\norm*{x - y}^2 \right\} \\
    ={}& \frac{r}{\ell} \sup_{y \in S} \min_{i = 1, \dots, m} \left\{ \nabla f_i(x)^\T \left( \frac{\ell}{r}(x - y) \right) + \frac{\ell}{r}(g_i(x) - g_i(y)) - \frac{r}{2} \left\| \frac{\ell}{r} (x - y) \right\|^2 \right\} \\
    \le{}& \frac{r}{\ell} \sup_{y \in S} \min_{i = 1, \dots, m} \left\{ \nabla f_i(x)^\T \left( \frac{\ell}{r} (x - y) \right) + g_i(x) - g_i \left(x - \frac{\ell}{r}(x - y)\right) \right. \\
        &\hspace{7em} \left. - \frac{r}{2}\left\|\frac{\ell}{r}(x - y)\right\|^2 \right\}
\end{align*}
where the first inequality follows from the convexity of~$g_i$.
Since~$S$ is convex, $x, y \in S$ implies~$x - (\ell / r)(x - y) \in S$.
Therefore, from the definition~\eqref{eq: w_ell} of~$w_r$, we get
\[
    w_\ell(x) \le \frac{r}{\ell} w_r(x)
.\] 
\end{proof}

Considering \Cref{rem: composite ref}~\ref{enum: w u correspond}, we get the following corollary.
\begin{corollary}
    Assume that each component~$F_i$ of the objective function~$F$ of~\cref{eq: CMOP} is convex.
    Recall that~$u_\ell$ is defined by~\cref{eq: u_ell} for all~$\ell > 0$.
    Let~$r$ be an arbitrary scalar such that $r \ge \ell$.
    Then, we get
    \[
        u_r(x) \le u_\ell(x) \le \frac{r}{\ell} u_r(x) \forallcondition{x \in S}
    .\]
\end{corollary}

\begin{remark}
    For unconstrained problems, we can consider the following inequality.
    \[
        w_L(x) \ge \tau u_0(x) \condition{for all~$x \in \setR^n$ for some~$\tau > 0$}
    ,\]
    which is an extension of the proximal-PL inequality for scalar optimization~\cite{Karimi2016}.
    Under this condition, we can prove that proximal gradient methods for multiobjective optimization~\cite{Tanabe2019} have linear convergence rate~\cite{Tanabe2023}.
    Note that this inequality holds particularly if each~$f_i$ is strongly convex from Theorem~\ref{thm: merit between}~\ref{enum: merit between convex} and Theorem~\ref{thm: merit inner}.
\end{remark}

\section{Level-boundedness of the proposed merit functions} \label{sec: level-bounded}
Recall that we call a function \emph{level-bounded} if every level set is bounded.
This is an important property because it ensures that the sequences generated by descent methods have accumulation points.
We state below sufficient conditions for the level-boundedness of the merit functions proposed in \cref{sec: merit}.
\begin{theorem}
    Consider~$u_0$,~$u_\ell$ and~$w_\ell$ defined in~\cref{eq: u_0,eq: u_ell,eq: w_ell}, respectively, for all~$\ell > 0$.
    Then, the following statements hold.
    \begin{enumerate}
        \item If~$F_i$ is level-bounded for all~$i = 1, \dots, m$, then~$u_0$ is level-bounded. \label{enum: u_0 level-bounded}
        \item If~$F_i$ is convex and level-bounded for all~$i = 1, \dots, m$, then~$u_\ell$ is level-bounded for all~$\ell > 0$. \label{enum: u_ell level-bounded}
        \item Suppose that~$F$ has the composite structure~\cref{eq: f + g}.
            If~$f_i$ is~$\mu_i$-convex for some~$\mu_i \in \setR$ or~$\nabla f_i$ is $L_i$-Lipschitz continuous for some~$L_i > 0$, and~$F_i$ is convex and level-bounded for all~$i = 1, \dots, m$, then~$w_\ell$ is level-bounded for all~$\ell > 0$. \label{enum: w level-bounded}
    \end{enumerate}
\end{theorem}
\begin{proof}
    \ref{enum: u_0 level-bounded}: Suppose, contrary to our claim, that~$u_0$ is not level-bounded.
    Then, there exists~$\alpha \in \setR$ such that~$\Set*{ x \in S }{ u_0(x) \le \alpha }$ is unbounded.
    By the definition~\eqref{eq: u_ell} of~$u_0$, the inequality~$u_0(x) \le \alpha$ can be written as
    \[
        \sup_{y \in S} \min_{i = 1, \dots, m} \{ F_i(x) - F_i(y) \} \le \alpha.
    \]
    This implies that for some fixed~$z \in S$, there exists~$j \in \set*{1, \dots, m}$ such that
    \[
        F_j(x) \le F_j(z) + \alpha.
    \]
    Therefore, it follows that
    \[
        \{ x \in S \mid u_0(x) \le \alpha \} \subseteq \bigcup_{j = 1}^m \{ x \in S \mid F_j(x) \le F_j(z) + \alpha \}.
    \]
    Since $F_i$ is level-bounded for all $i = 1, \dots, m$, the right-hand side must be bounded, which contradicts the unboundedness of the left-hand side.

    \ref{enum: u_ell level-bounded}: Recall the definitions~\cref{eq: simplex,eq: Moreau envelope,eq: proximal operator} of~$\Delta^m,\envelope$, and~$\prox$.
    \Cref{eq: u dual} gives
    \[
        \begin{split}
            u_\ell(x) &= \min_{\lambda \in \Delta^m} \left\{ \sum_{i = 1}^{m} \lambda_i F_i(x) - \ell \envelope_{\frac{1}{\ell} \sum_{i = 1}^{m} \lambda_i F_i + \indicator_S}(x) \right\} \\
                      &\begin{multlined}
                          = \min_{\lambda \in \Delta^m} \sum_{i = 1}^{m} \lambda_i \left\{ F_i(x) - F_i\left( \prox_{\frac{1}{\ell} \sum_{i = 1}^{m} \lambda_i F_i + \indicator_S}(x) \right) \right. \\
                          \left.- \frac{\ell}{2} \norm*{x - \prox_{\frac{1}{\ell} \sum_{i = 1}^{m} \lambda_i F_i + \indicator_S}(x)}^2 \right\} 
                      \end{multlined} \\
                      &\ge \frac{1}{2} \min_{\lambda \in \Delta^m} \sum_{i = 1}^{m} \lambda_i \left\{ F_i(x) - F_i\left( \prox_{\frac{1}{\ell} \sum_{i = 1}^{m} \lambda_i F_i + \indicator_S}(x) \right) \right\} \\
                      &= \frac{1}{2} \min_{i = 1, \dots, m} \left\{ F_i(x) - F_i\left( \prox_{\frac{1}{\ell} \sum_{i = 1}^{m} \lambda_i F_i + \indicator_S}(x) \right) \right\} 
        ,\end{split}
    \] 
    where the inequality follows from \cref{cor: second prox}.
    Therefore, with similar arguments given in the proof of statement~\ref{enum: u_0 level-bounded}, we can show the level-boundedness of~$u_\ell$ by contradiction.

    \ref{enum: w level-bounded}: From \cref{thm: merit between,thm: merit inner}, there exist some~$a > 0$ and~$r > 0$ such that~$u_r(x) \le a w_\ell(x)$ for all~$x \in S$.
    Since statement~\ref{enum: u_ell level-bounded} implies that~$u_r$ is level-bounded,~$w_\ell$ is also level-bounded.
\end{proof}
As indicated by the following example, our proposed merit functions are not necessarily level-bounded even if $F$ is level-bounded.
\begin{example}
    Consider the bi-objective function $F \colon \setR \to \setR^2$ with each component given by
    \[
        F_1(x) \coloneqq x^2, \quad F_2(x) \coloneqq 0.
    \]
    Then, the merit function~$u_0$ defined by~\eqref{eq: u_ell} is written as
    \begin{align*}
        u_0(x) & = \sup_{y \in \setR} \min \{ F_1(x) - F_1(y), F_2(x) - F_2(y) \} \\
             & = \sup_{y \in \setR} \min \{ (x^2 - y^2), 0 \} = 0.
    \end{align*}
    On the other hand, $F$ is level-bounded because $\lim_{\|x\| \to \infty} F_1(x) = \infty$.
\end{example}

\section{The multiobjective proximal-PL inequality and error bounds} \label{sec: error bound}
In this section, we extend the proximal-PL inequality introduced in~\cite{Karimi2016} and shows that it induces the proposed merit function's error bound.
This section assumes that~\cref{eq: CMOP} has the composite structure~\cref{eq: f + g}; note that if~$f = 0$ or~$g = 0$, the assumption also holds for~\cref{sec: merit continuous,sec: merit convex}.

We first define the \emph{multiobjective proximal-PL inequality}.
\begin{definition} \label{def:proximal_PL_MO}
    Assume that~$f_i$ is~$L_i$-Lipschitz continuous with~$L_i > 0$ for all~$i = 1, \dots, m$ and let~$L \coloneqq \max_{i = 1, \dots, m} L_i$.
    We say that~\cref{eq: CMOP} satisfies the multiobjective proximal-PL inequality if there exists~$\tau > 0$ such that
    \begin{equation} \label{eq:proximal_PL_MO}
        w_L(x) \ge \tau u_0(x) \forallcondition{x \in S}
    \end{equation}
    with~$u_0$ and~$w_L$ given by~\cref{eq: u_0,eq: w_ell}.
\end{definition}
If~$m = 1$, \cref{eq:proximal_PL_MO} reduces to the proximal-PL inequality~\cite{Karimi2016}.

We state below some sufficient conditions for~\cref{eq:proximal_PL_MO}.
\begin{proposition} \label{thm:proximal_PL_MO_sufficient}
    \begin{enumerate}
        \item When~$f_i$ is $\mu_i$-convex with~$\mu_i > 0$,~\cref{eq:proximal_PL_MO} holds with~$\tau \coloneqq \min (\mu / L, 1)$, where~$\mu \coloneqq \min_{i = 1, \dots, m} \mu_i$. \label{thm:proximal_PL_MO_sufficient:sc}
        \item Assume that~$f_i(x) \coloneqq h(A_i x)$ with some strongly convex function~$h_i$ and linear transformation~$A_i$,~$g_i \coloneqq 0$, and~$S = \mathcal{X}$ is a polyhedral set.
              If each~$\min_{x \in S} F_i(x)$ has a nonempty set~$X_i^*$ for~$i = 1, \dots, m$, then~\cref{eq:proximal_PL_MO} holds with some constant~$\tau$. \label{thm:proximal_PL_MO_sufficient:sc_linear}
    \end{enumerate}
\end{proposition}
\begin{proof}
    \ref{thm:proximal_PL_MO_sufficient:sc}:
    Since~$f_i$ is strongly convex, \cref{thm: merit between}~\ref{enum: merit between convex} gives
    \begin{equation}
        u_0(x) \le w_\mu(x) \forallcondition{x \in S}
        .\end{equation}
    Applying \cref{thm: merit inner} to the above inequality implies
    \begin{equation}
        u_0(x) \le \max \left( \frac{L}{\mu}, 1 \right) w_{L}(x) \forallcondition{x \in S}
        ,\end{equation}
    which means
    \begin{equation}
        w_{L}(x) \ge \min \left( \frac{\mu}{L}, 1 \right) u_0(x) \forallcondition{x \in S}
        .\end{equation}

    \ref{thm:proximal_PL_MO_sufficient:sc_linear}:
    Since~$S$ is polyhedral, we can write it as~$\Set{x \in \setR^n}{B x \le c}$ for some matrix~$B$ and vector~$c$.
    We now show that for all~$i = 1, \dots, m$ there exists some~$z_i$ such that
    \[
        X_i^* = \Set{x \in \setR^n}{B x \le c \text{ and } A_i x = z_i}
    .\]
    To obtain a contradiction, suppose that there exists~$x^1 \in X_i^*$ and~$x^2 \in X_i^*$ such that~$A_i x^1 \neq A_i x^2$.
    Clearly, we have~$f_i(x^1) = f_i(x^2)$.
    Since~$h_i$ is strongly convex, we get
    \begin{align*}
        f_i(x^1) & = \frac{1}{2} f_i(x^1) + \frac{1}{2} f_i(x^2) = \frac{1}{2} h_i(A_i x^1) + \frac{1}{2} h_i(A_i x^2)                            \\
                 & > h_i\left( A_i \left( \frac{1}{2} x^1 + \frac{1}{2} x^2 \right) \right) = f_i\left( \frac{1}{2} x^1 + \frac{1}{2} x^2 \right)
        ,\end{align*}
    which contradicts the fact that~$x^1 \in X_i^*$.
    Therefore, we can use Hoffman's error bound~\cite{Hoffman1952}, and so there exists some~$\rho_i > 0$ such that for any~$x \in S$, there exists~$x_i^* \in X_i^*$ with
    \[
        \norm*{x - x_i^*} \le \rho_i \norm*{\max \left[ \begin{pmatrix} B \\ A_i \\ - A_i \end{pmatrix} x - \begin{pmatrix} c \\ z_i \\ - z_i \end{pmatrix}, 0 \right] }
    .\]
    Note that we take the $\max$ operator componentwise on the right-hand side.
    Since~$B x - c \le 0$ for all~$x \in S$, we have
    \[
        \norm*{x - x_i^*} \le \rho_i \norm*{\max \left[ \begin{pmatrix} A_i \\ - A_i \end{pmatrix} x - \begin{pmatrix} z_i \\ - z_i \end{pmatrix} , 0 \right] } \forallcondition{x \in S}
    ,\]
    which yields
    \[
        \norm*{x - x_i^*}^2 \le \rho_i^2 \norm*{A_i x - z_i}^2 \forallcondition{x \in S}
    .\]
    Since~$\proj_{X_i^*}(x) \in X_i^*$, it follows that
    \begin{equation} \label{eq:proximal_PL_MO_sufficient:sc_linear:linear_trans_ineq}
        \norm*{x - \proj_{X_i^*}(x)}^2 \le \norm*{x - x_i^*}^2 \le \rho_i^2 \norm*{A_i \left(x - \proj_{X_i^*}(x)\right)}^2 \forallcondition{x \in S}
        .\end{equation}
    Now, suppose that~$x \in S$.
    From the definition~\cref{eq: u_0} of~$u_0$, we get
    \begin{align*}
        u_0(x) & = \sup_{z \in S} \min_{i = 1, \dots, m} [F_i(x) - F_i(z)]                                                                                       \\
                    & \le \min_{i = 1, \dots, m} \sup_{z \in S} [F_i(x) - F_i(z)] = \min_{i = 1, \dots, m} \left[ F_i(x) - F_i\left( \proj_{X_i^*}(x) \right) \right]
        ,\end{align*}
    where the second equality holds because~$\proj_{X_i^*}(x) = \argmin_{x \in S} F_i(x)$.
    Assuming that~$h_i$ is $\sigma_i$-convex with~$\sigma_i > 0$, it follows that
    \begin{align*}
        u_0(x) & \le \min_{i = 1, \dots, m} \left[ \nabla h_i(A_i x)^\T A_i \left(x - \proj_{X_i^*}(x)\right) - \frac{\sigma_i}{2} \norm*{A_i\left( x - \proj_{X_i^*}(x) \right) }^2 \right] \\
                    & = \min_{i = 1, \dots, m} \left[ \nabla f_i(x)^\T \left(x - \proj_{X_i^*}(x) \right) - \frac{\sigma_i}{2} \norm*{A_i \left( x - \proj_{X_i^*}(x) \right) }^2 \right].
    \end{align*}
    Applying~\cref{eq:proximal_PL_MO_sufficient:sc_linear:linear_trans_ineq} to the above inequality leads to
    \[
        u_0(x) \le \min_{i = 1, \dots, m}
            \Biggl[ \nabla f_i(x)^\T (x - \proj_{X_i^*}(x)) - \frac{\sigma_i}{2 \rho_i^2} \norm*{x - \proj_{X_i^*}(x)}^2 \Biggr].
    \]
    Let~$e \in \Delta^m$ with~$\Delta^m$ given by~\cref{eq: simplex}.
    Since~$\min_{i = 1, \dots, m} v_i = \min_{e \in \Delta^m} \sum_{i = 1}^{m} e_i v_i$ for any~$v \in \setR^m$, we get
    \begin{align*}
        u_0(x) & \le \min_{e \in \Delta^m} \sum_{i = 1}^{m} e_i
            \left[ \nabla f_i(x)^\T (x - \proj_{X_i^*}(x)) - \frac{\sigma_i}{2 \rho_i^2} \norm*{x - \proj_{X_i^*}(x)}^2 \right] \\
                    & \le \min_{e \in \Delta^m} \sup_{z \in \setR^n} \sum_{i = 1}^{m} e_i \left[ \nabla f_i(x)^\T (x - z) - \frac{\sigma_i}{2 \rho_i^2} \norm{x - z}^2 \right] \\
                    & = \sup_{z \in S} \min_{e \in \Delta^m} \sum_{i = 1}^{m} e_i \left[ \nabla f_i(x)^\T (x - z) - \frac{\sigma_i}{2 \rho_i^2} \norm{x - z}^2 \right]   \\
                    & = \sup_{z \in S} \min_{i = 1, \dots, m} \left[ \nabla f_i(x)^\T (x - z) - \frac{\sigma_i}{2 \rho_i^2} \norm{x - z}^2 \right]                         \\
                    & \le w_{\min\limits_{i = 1, \dots, m} \sigma_i / \rho_i^2}(x)
        ,\end{align*}
    where the first equality follows from the Sion's minimax theorem~\cite{Sion1958}, and the third equality comes from the definition~\cref{eq: w_ell} of~$w_{\rho_i^2 / \min\limits_{i = 1, \dots, m} \sigma_i}$.
    Thus, \cref{thm: merit inner} gives
    \[
        u_0(x) \le \max \left( \frac{L \rho_i^2}{\min_{i = 1, \dots, m} \sigma_i} , 1 \right) w_L(x)
    ,\]
    which completes the proof.
\end{proof}

We now show that the multiobjective proximal-PL inequality~\cref{eq:proximal_PL_MO} leads to the error-bound property.
\begin{theorem} \label{thm:proximal_PL_MO_EB}
    Let~$x \in S$.
    Suppose that~$f_i$ is~$L_i$-smooth with~$L_i > 0$ for each~$i = 1, \dots, m$,~$L \coloneqq \max_{i = 1, \dots, m} L_i$, and the multiobjective proximal-PL inequality~\cref{eq:proximal_PL_MO} holds with~$\tau > 0$.
    Then, the trajectory~$\set*{W_{L}^k(x) \coloneqq \overbrace{W_{L} \circ \dots \circ W_{L}}^{m}(x)}$ converges linearly to a weakly Pareto optimal point~$x^*$ and
    \begin{equation}
        u_0(x) \ge \frac{\tau L}{8} \norm{x - x^*}^2 \ge \frac{\tau L}{8} \min_{z \in X^*} \norm{x - z}^2
        ,\end{equation}
    where~$u_0$ and~$W_{L}$ are given by~\cref{eq: u_0,eq: W_ell}, respectively, and~$X^*$ denotes the set of weakly Pareto optimal solutions.
\end{theorem}
\begin{proof}
    Recall that~$u_0$ is non-negative due to \cref{thm: u}.
    We have
    \[
        \sqrt{u_0(x)} - \sqrt{u_0\left(W_{L}(x)\right)} = \frac{u_0(x) - u_0\left(W_{L}(x)\right)}{\sqrt{u_0(x)} + \sqrt{u_0\left(W_{L}(x)\right)}}
    .\]
    The definition~\cref{eq: u_0} of~$u_0$ gives
    \[
        u_0(x) - u_0\left(W_{L}(x)\right) \ge \min_{i = 1, \dots, m} \left[F_i(x) - F_i\left(W_{L}(x)\right)\right] \ge w_{L}(x)
    ,\]
    where the second inequality follows from the descent lemma~\cite[Proposition A.24]{Bertsekas1999},~\cref{eq: w_ell}, and \cref{eq: W_ell}.
    Note that this inequality, together with~\cref{eq:proximal_PL_MO}, proves that~$\set*{W_{L}^k(x)}$ converges linearly to zero.
    On the other hand, since~$u_0(x) \ge u_0\left(W_{L}(x)\right)$ because of \cref{thm: w ell} and the above inequality, we get
    \[
        \sqrt{u_0(x)} + \sqrt{u_0\left(W_{L}(x)\right)} \le 2 \sqrt{u_0(x)} \le 2 \sqrt{w_{L}(x) / \tau}
    ,\]
    where the second inequality comes from~\cref{eq:proximal_PL_MO}.
    Then, the above three inequalities show
    \[
        \sqrt{u_0(x)} - \sqrt{u_0\left(W_{L}(x)\right)} \ge \frac{w_{L}(x)}{2 \sqrt{w_{L}(x) / \tau}} = \frac{1}{2} \sqrt{\tau w_{L}(x)}
    .\]
    Therefore, it follows from~\cref{eq:reg_lin_gap_MO_LB} that
    \[
        \sqrt{u_0(x)} - \sqrt{u_0\left(W_{L}(x)\right)} \ge \frac{\sqrt{\tau L}}{2 \sqrt{2}} \norm*{x - W_{L}(x)}
    .\]
    More generally, we arrive at
    \[
        \sqrt{u_0\left(W_{L}^k(x)\right)} - \sqrt{u_0\left(W_{L}^{k + 1}(x)\right)} \ge \frac{\sqrt{\tau L}}{2 \sqrt{2}} \norm*{W_{L}^k(x) - W_{L}^{k + 1}(x)}
    \]
    for all~$k = 0, 1, \dots$.
    Adding up the above inequality from~$k = k_1$ to~$k = k_2 - 1$ yields
    \[
        \sqrt{u_0\left(W_{L}^{k_1}(x)\right)} - \sqrt{u_0\left(W_{L}^{k_2}(x)\right)} \ge \frac{\sqrt{\tau L}}{2 \sqrt{2}} \sum_{k = k_1}^{k_2 - 1} \norm*{W_{L}^k(x) - W_{L}^{k + 1}(x)}
    .\]
    Thus, the triangle inequality implies
    \begin{equation} \label{eq:proximal_PL_MO_EB:sq_gap_Cauchy}
        \sqrt{u_0\left(W_{L}^{k_1}(x)\right)} - \sqrt{u_0\left(W_{L}^{k_2}(x)\right)} \ge \frac{\sqrt{\tau L}}{2 \sqrt{2}} \norm*{W_{L}^{k_1}(x) - W_{L}^{k_2}(x)}
        .\end{equation}
    As~$k_1, k_2 \to \infty$, the left-hand side tends to zero.
    Therefore, the right-hand side also tends to zero because of the non-negativity of the norm.
    This means that~$\set{W_{L}^k(x)}$ is the Cauchy sequence, which is convergent to some weakly Pareto optimal point~$x^*$.
    Substituting~$k_1 = 0$ and~$k_2 = \infty$ into~\cref{eq:proximal_PL_MO_EB:sq_gap_Cauchy} leads to
    \[
        \sqrt{u_0(x)} \ge \frac{\sqrt{\tau L}}{2 \sqrt{2}} \norm{x - x^*}
    .\]
\end{proof}
This theorem also presents the error-bound property of~$w_\ell$ and~$u_\ell$ for any~$\ell > 0$ because of~\cref{eq:proximal_PL_MO}, \cref{thm: merit inner}, and \cref{thm: merit between}~\ref{enum: merit between Lipschitz}.

\section{Conclusion} \label{sec: conclusion}
We first proposed a simple merit function for~\cref{eq: CMOP} in the sense of weak Pareto optimality and showed its lower semicontinuity.
We also defined a regularized merit function when~$F$ is convex and discussed its continuity, the way of evaluating it, its differentiability, and the properties of its stationary points.
Furthermore, when each~$F_i$ is composite, we introduced a regularized and partially linearized merit function in the sense of Pareto stationarity and showed similar properties.
In addition, we gave sufficient conditions for the proposed merit functions to be level-bounded and to provide error bounds.

We can consider a natural extension of our proposed merit functions for vector problems with an infinite number of objective functions.
We can also regard the generalization of other merit functions for scalar problems, such as the implicit Lagrangian~\cite{Mangasarian1993} and the squared Fischer-Burmeister function~\cite{Kanzow1996}, to multiobjective and vector problems.
These will be some subjects for future works.

\section*{Funding}
    This work was supported by the Grant-in-Aid for Scientific Research (C) (21K11769 and 19K11840) and Grant-in-Aid for JSPS Fellows (20J21961) from Japan Society for the Promotion of Science.

\bibliographystyle{tandfx}
\bibliography{library}

\end{document}